\documentclass{article}
\usepackage{authblk} %

\usepackage{amsmath,amssymb}
\usepackage{graphicx}
\usepackage{hyperref}

\usepackage{amsmath, amssymb, amsfonts}
\usepackage{mathtools}
\usepackage{tikz}
\usetikzlibrary{shapes.geometric, arrows.meta, positioning}
\usepackage[applemac]{inputenc}
\usepackage{bm}

\usepackage{booktabs}   

\renewcommand{\arraystretch}{0.82}

\usepackage{etoolbox}
\BeforeBeginEnvironment{tabular}{\footnotesize}
\AfterEndEnvironment{tabular}{\normalsize}

\newcommand{\sbm}[1]{\mbox{\scriptsize\boldmath $#1$}}
\newcommand{\tras}{^{\mbox{\scriptsize tr}}}
\newcommand{\E}{\mbox{\bf E}}
\newcommand{\var}{\mbox{\bf Var}}
\newcommand{\cov}{\mbox{\bf Cov}}
\newcommand{\N}{\mbox{\bf N}}
\newcommand{\scN}{\mbox{\scriptsize \bf N}}
\newcommand{\R}{\mbox{\bf R}}

\newcommand{\pr}{\mbox{\bf P}}
\newcommand{\un}{\mbox{\bf 1}}
\newcommand{\bec}{\begin{equation}}
\newcommand{\eec}{\end{equation}}

\newtheorem{theorem}{Theorem}
\newtheorem{definition}{Definition}

\title{Multivariate EDF tests for uniformity, normality, spherical and elliptical symetry, and independence based on a Brownian sheet deconstruction}

\author[1]{Alejandra Caba\~na }
\author[2]{Enrique M. Caba\~na}

\affil[1]{Universitat Aut\`onoma de Barcelona, Spain}
\affil[2]{PEDECIBA, Uruguay}
\date{}

\begin{document}
\maketitle

\begin{abstract}
This paper extends a recently proposed family of EDF-based goodness-of-fit procedures for the hypercube $[0,1]^p$ - the m-test and the s-test - which are based on a unique deconstruction of the $p$-parameter Brownian sheet into independent Gaussian processes. 

We use the fact that whenever a null hypothesis implies a joint distribution that factorizes into independent continuous components after a suitable mapping, the problem can be reduced to a uniformity test on the hypercube via componentwise probability integral transforms.
Specifically, we introduce and analyze new procedures derived from these principles for testing uniformity on the hypersphere $S^p$, as well as multivariate normality, spherical and elliptical symmetry, and independence in $R^p$. The methodology is based on the decomposition of finite signed measures into zero-marginal components to isolate coordinate interactions. Empirical power comparisons show that these extended procedures are highly competitive with existing methods in the statistical literature, demonstrating particular sensitivity to coordinate-based dependencies and joint dependency structures.
\end{abstract}


\section{Introduction}

While basing statistical decisions on the distance between the sample EDF and the theoretical distribution under the null hypothesis is a conceptually robust approach, it remains significantly underused in multivariate settings. The primary barriers to this strategy lie in the fact that the null distributions of such statistics 
depend heavily on both the dimension $p$ and the sample size $n$, while their asymptotic forms are often mathematically complex and difficult to implement.
In a recent article \cite{cabana2025}, we addressed these challenges by taking profit of the fact that a $p$-parameter Brownian sheet can be uniquely decomposed into a sum of $2^p$ independent Gaussian processes. This 
suggests a new framework for constructing uniformity tests on the hypercube $[0,1]^p$ for $p\geq1$ based on the asymptotic decomposition of the EDF. These procedures, which we refer to as the \emph{m-test} and the \emph{s-test}, use statistics whose asymptotic distributions are well-defined. However, in practice, we use Monte Carlo simulations to approximate critical regions. This approach not only simplifies the computational burden but also significantly improves finite-sample performance compared to methods relying strictly on asymptotic approximations.
The methodology developed here comprises two complementary structural components.

The first is probabilistic and is based on the unique decomposition of finite signed measures on product spaces into components with zero marginals, indexed by subsets $H\subset J:=\{1,2,…,p\}$. 
When applied to the Wiener measure, this allows the Brownian sheet to be expressed as a sum of independent Gaussian terms related to independent Brownian pillow measures. This yields a family of statistics based on the squared $L^2$ norms of these measures (see equation \ref{bnH})
which converge, under the null hypothesis of uniformity, to mutually independent components and constitute the basis of the proposed m- and s-tests

The second component exploits the fact that, whenever the null hypothesis implies that the joint distribution can be mapped to a product of independent continuous marginals, the problem reduces to testing uniformity on  $[0,1]^p$ via componentwise probability integral transforms.

By applying the first component to provide the test statistics for the hypercube and the second to map complex multivariate hypotheses back to the hypercube, we extend the applicability of these tests to a much broader class of problems.

In this article, we show that straightforward applications of these principles lead to highly competitive procedures for testing uniformity on the hypersphere $S^p$, as well as multivariate normality, spherical and elliptical symmetry, and independence in $R^p$. Empirical power comparisons demonstrate that these tests are particularly effective at detecting non-standard deviations, such as coordinate-based interactions in copula models, where traditional distance-based and directional tests frequently lose sensitivity.

An R package \cite{MuniCandS} for computing the proposed tests is available.

\section{Decomposition of a signed measure on a product probability space as a sum of zero-marginals components}

Let $(\Omega,{\cal A},P)=\left(\prod_{j=1}^p\Omega_j,\bigotimes_{j=1}^p \mathcal{A}_j, \prod_{j=1}^pP_j\right)$ denote a product probability space, and $\mu$ a finite signed measure on $(\Omega,{\cal A})$. Recall that any measure on $(\Omega,{\cal A})$ is determined by the measures of the product sets $\prod_{j=1}^pA_j$ for $A_j\in {\cal A}_j, j=1,2,\dots,p$.

Given a subset $H$ of $J:=\{1,2,\dots,p\}$, we denote $\Omega_H=\prod_{j\in H}\Omega_j$ and $P_H=\prod_{j\in H}P_j$. A product of measurable sets $A=\prod_{j=1}^pA_j$ is said to be an $H$-set when $A_j=\Omega_j$ for $j\not\in H$.

A measure $M_H$ on $\Omega_H$ is said to be a zm-measure when all its marginals vanish, that is, when $M(\prod _{j\in H}A_j)=0$ whenever $A_j=\Omega_j$ for at least some $j\in H$.

The following decomposition holds:

\begin{theorem} Any finite signed measure $\mu$ on $\Omega$ can be written as the sum
\bec\label{ladesco}\mu\left(\prod_{j=1}^pA_j\right)=\sum_{H\subset J} P_{J\setminus H}\left(\prod_{j\in\{J\setminus H\} }A_j\right)  \times \mu_H \left(\prod_{j\in H}A_j\right)\eec
where each $\mu_H$ is a signed zm-measure on $\Omega_H$.

The decomposition is unique, and the map $\mu\mapsto (\mu_H)_{H\subset J}$ is linear and continuous in the supremum norm.
\end{theorem}

A version of this result formulated in terms of the distribution function -  specifically for \(\Omega_j = [0,1]\) and \(P_{j}\) being the Lebesgue measure (\(j=1,\dots,p\)) - is contained in \cite{cabana2025}. We include a proof of the current statement in an appendix (\S\ref{appe}) for easy reference.

\section{Decomposing  the standard Wiener measure and the empirical process}

Let $W$ be the standard $p$-parameter Wiener measure
on $\Omega=\prod_{j=1}^p\Omega_j$, where $\Omega_j=[0,1]$ for $j=1,2,\dots,p$.
The following statement establishes that $W$ 
can be uniquely represented as a sum of independent Gaussian components, each associated with a specific subset of coordinates.

\begin{theorem}
The standard $p$-parametric Gaussian Wiener measure $W$ on $[0,1]^p$ can be decomposed into a linear combination of independent Brownian pillows $b_H$ associated with each subset of parameters $H \subseteq J := \{1, \dots, p\}$:
\begin{equation}\label{asum}
W\left(\prod_{j=1}^p A_j\right) = \sum_{H\subset J} \left( \prod_{j \in J\setminus H} \lambda(A_j) \right) \times b_H\left(\prod_{j\in H} A_j\right)
\end{equation}
where $\lambda$ denotes the Lebesgue measure and each $b_H$ is a centred Gaussian zm-measure with covariance \[\E b_H\left(\prod_{j\in H} A_j\right)b_H\left(\prod_{j\in H} B_j\right)=\prod_{j\in H} \left( \lambda(A_j\cap B_j) - \lambda(A_j)\lambda(B_j) \right).\]

\end{theorem}

\paragraph{Proof}
Recall that the covariance of the standard Wiener measure $W$ on $[0,1]^p$ is given by the product of the marginal covariances:
\[ \E \left[ W\left(\prod_{j=1}^p A_j\right) W\left(\prod_{j=1}^p B_j\right) \right] = \prod_{j=1}^p \lambda(A_j \cap B_j). \]
To reveal the underlying additive structure, we use the algebraic identity for the product of binomials \begin{equation}\label{prodbin}\prod_{j=1}^p(1+\alpha_j) = \sum_{H\subset J} \prod_{j\in H} \alpha_j.\end{equation} By setting $\alpha_j = \frac{\lambda(A_j\cap B_j)}{\lambda(A_j)\lambda(B_j)} - 1$,  the covariance of $W$ can be written as:
\begin{equation} \label{cov_expand}
\sum_{H\subset J} \left( \prod_{j\in J\setminus H} \lambda(A_j)\lambda(B_j) \right) \times \prod_{j\in H} \bigl( \lambda(A_j\cap B_j) - \lambda(A_j)\lambda(B_j) \bigr).
\end{equation}
The $H$-th term in this sum is the covariance of the product
\[\prod_{j \in J\setminus H}\lambda(A_j)\times b_H\left(\prod_{j\in H}A_j\right)\]
where $b_H$ denotes a measure on $\Omega_H=\prod_{j\in H}\Omega_j$, with covariances 
\begin{equation*} \E b_H\left(\prod_{j\in H}A_j)\right)b_H\left(\prod_{j\in H}B_j\right)=\prod_{j\in H}(\lambda(A_j\cap B_j)-\lambda(A_j)\lambda(B_j)).\end{equation*}

Because the joint covariance in \eqref{cov_expand} is a sum of these terms with no cross-covariance between different subsets $H$, the components of the decomposition are
 uncorrelated.
 
 The variance of $b_H$ vanishes whenever at least one $A_j$ is $\Omega_j$, so that $b_H$ is almost surely a zero-marginal measure, hence  (\ref{cov_expand}) is the canonical decomposition established in Theorem 1.
 
Moreover, since the original Wiener measure W is a centred Gaussian measure and the decomposition is linear, the resulting components $b_H$ are also Gaussian. 
 The $b_H$ (and their c.d.f.'s) are known as Brownian pillow measures (and processes) in the probability literature. They are the result of conditioning the Wiener sheet $W_H$ on $\Omega_H$ to be a zm-measure.\hfill\framebox{$ $}

\medskip

In order to describe distributions associated to $b_H$, we will obtain a representation of its c.d.f. $b_H((t_j)_{j\in H})=b_H(\prod_{j\in H}[0,t_j])$ that we denote by the same symbol.

\medskip

 The Brownian pillow $b_H$ reduces to a standard Brownian bridge for $\#H=1$. The well known Karhunen-Loeve expansion of the Brownian bridge in terms of the eigenfunctions
  of the covariance kernel
 $b(x)=\sum_{\nu\in\scN}\sqrt{\lambda_{\nu}}Z_{\nu}\psi_{\nu}(x)$, $ \psi_{\nu}(x)=\sqrt 2\sin(\nu\pi x)$, $\lambda_{\nu}=\frac{1}{\nu^2\pi^2}$, $\nu\in \N:=\{1,2,3,\dots\}$
 where $\{Z_{\nu}:\nu\in\bm N\}$ are  i.i.d. standard Gaussian variables (see for instance \cite{Durbin1973}) 
 is easily extended to the Brownian pillow due to the factorization of the covariance, thus yelding
 \begin{equation}\label{KL}b_H((t_j)_{j\in H})=\sum_{\sbm\nu\in\scN^{\#H}}Z_{\sbm\nu}\prod_{j\in H}\left(\sqrt 2\frac{\sin(\nu_j\pi t_j)}{\nu_j\pi}\right),\end{equation} 
 with $\{Z_{\sbm\nu}:\bm\nu\in\N^{\#H}\}$ i.i.d. standard Gaussian,  and this leads to an expansion of the squared norm of $b_H$:
 \begin{equation}\label{nortent}\|b_H\|^2=\sum_{\sbm\nu\in\scN^{\#H}}\frac{Z_{\sbm\nu}^2}{(\prod_{j\in H}\bm\nu_j\pi)^2}.\end{equation}
 This expansion allows to compute its expectation $\E \|b_H\|^2$ $=(\sum_{\nu\in\sbm N}\frac{1}{\nu^2\pi^2})^{\#H}$ $=\frac{1}{6^{\#H}}$
 and variance $\var \|b_H\|^2=2(\sum_{\nu\in\sbm N}\frac{1}{\nu^4\pi^4})^{\#H}$ $=\frac{2}{90^{\#H}}$  which show that for different cardinality of $H$ the squared norms of $b_H$ are of very different order of magnitude.
 
 \medskip

As for the unit measure $\delta_{\sbm X}(\prod_{j=1}^pA_j)=\prod_{j=1}^p\delta_{X_j}(A_j)$ concentrated at $\bm X=(X_1,\dots,X_p)\tras$, we set $\alpha_j=\left(\frac{\delta_{X_j}(A_j)}{\lambda(A_j)}-1\right)$ in (\ref{prodbin}) and get the canonical decomposition of Theorem 1
\[\prod_{j=1}^p\delta_{X_i}(A_j)=\sum_{H\subset J}\prod_{j\in J\setminus H}\lambda(A_j)\prod_{j\in H}\left(\delta_{X_j}(A_j)-\lambda(A_j)\right).\] 

 By the linearity, the measure associated to the uniform empirical process of an i.i.d. sample
${\cal X}_n=\{\bm X_1,\bm X_2,\dots,\bm X_n\}$ of vectors $\bm X_i=(X_{i,1},\dots,X_{i,p})\tras$ in $[0,1]^p$
has the decomposition \[W_n\left(\prod_{j=1}^pA_j\right)=\frac{1}{n}\sum_{i=1}^n\left(\prod_{j=1}^p\delta_{X_{i,j}}(A_j)-\prod_{j=1}^p\lambda(A_j)\right)\] \[=\frac{1}{n}\sum_{i=1}^n\sum_{H\subset J,H\not=\emptyset}\prod_{j\in J\setminus H}\lambda(A_j)\prod_{j\in H}\left(\delta_{X_j}(A_j)-\lambda(A_j)\right)\]
\begin{equation}\label{lala}=\sum_{H\subset J,H\not=\emptyset}\left(\prod_{j\in J\setminus H}\lambda(A_j)\right)\times b_{n,H}\left(\prod_{j\in H}A_j, \right)
\end{equation}
with \[b_{n,H}=\frac{1}{n}\sum_{i=1}^n\prod_{j\in H}\left(\delta_{X_j}(A_j)-\lambda(A_j)\right).\]

\section{Description of the m- test and the s-test}

\subsection{The test statistics}

The tests introduced in \cite{cabana2025} are similar to Cramér - von Mises test, but instead of relying on the quadratic distance between the EDF and the null distribution,  they are based on the distances between the zero marginals terms of their respective decompositions.

More precisely, our test statistic is not the $L^2$ norm of the empirical process
$W_n(\bm t)$.
 We use instead the quadratic norms \begin{equation}\label{bnH}\|b_{n,H}\|^2=\int_{\Omega_H} b_{n,H}(t_j)_{j\in H} \prod_{j\in H}dt_j\end{equation} of the processes $b_{n,H}$ appearing in (\ref{lala}).

The integrations required to compute these squared norms are specially simple because the integrands are factorized and the integrals commute with the product. The resulting expression is  
\[ \|b_{n,H}\|^2= \frac{1}{n} \sum_{h,i=1}^n \prod_{j \in H} \left( \frac{X_{h,j}^2 + X_{i,j}^2}{2} - X_{h,j} \vee X_{i,j} + \frac{1}{3} \right) .\]

If the distribution $F$ of the sample is the uniform distribution, the empirical process $W_n$ converges weakly to the pinned Brownian sheet
\[
W_1 \sim \sum_{H \subset J,\, H \neq \emptyset} W_H.
\]
Consequently, each $b_{n,H}$ converges in distribution to the corresponding Brownian pillow, and the vector of squared $L^2$ norms
\begin{equation}\label{defstat}
\left(\|b_{n,H}\|^2\right)_{H \subset J}
= \left(\int_{\Omega_H} b_{n,H}^2(\bm{t})\, d\bm{t}\right)_{H \subset J}
\end{equation}
converges to
$\left(\int_{\Omega_H} b_H^2(\bm{t})\, d\bm{t}\right)_{H \subset J}$, whose components are independent.

\medskip

If $F$ is not uniform, then $\|W_n\|^2 \to \infty$ almost surely, implying that at least one of the $2^p-1$ statistics $\|b_{n,H}\|^2$ diverges. This dichotomy ensures that any test rejecting the null hypothesis of uniformity when either the maximum or the sum of the $\|b_{n,H}\|^2$ is large is consistent.

We have already pointed out that for different cardinality of $H$ the corresponding $\|b_{n,H}\|^2$ are of different order of magnitud.  Therefore, to balance the weight given to deviations relative to the uniformity of each component in the decomposition of the empirical process, we aim to replace the squared norms with their p-values $p_{n,H}$ approximately equal to the asymptotic value $1-P_H(b_{n,H}).$

Although the distribution of $\|b_H\|^2$ under the null hypothesis of uniformity is known (see (\ref{nortent})), its representation as an infinite weighted sum of independent $\chi^2_1$ variables makes direct numerical evaluation impractical. For this reason, we approximate the finite-sample p-values using Monte Carlo simulation.  This approach avoids the numerical computation of infinite series and, importantly, yields estimates of the actual finite-sample p-values rather than the asymptotic ones.

Based on these ideas, we introduce two classes of tests:

\begin{itemize}
\item[\textbf{m-test}] 
Let ${\cal H}$ be the family of all nonempty subsets of $J$. For each $H\in{\cal H}$, let $P_{n,H}$ be the c.d.f. of $\|b_{n,H}\|^2$ under the null hypothesis (that actually depends on $H$ through its cardinal $\#H$) and $p_{n,H}=1-P_{n,H}(\|b_{n,H}\|^2)$ the p-value of the observed statistic $\|b_{n,H}\|^2$, so that the vector $\bm p_{n,{\cal H}}\in \bm {R} ^{\#{\cal H}}$ with components $(p_{n,H})_{H \in {\cal H}}$ is approximately uniform on $[0,1]^{\#\cal H}$ for uniform samples, since it has uniform components asymptotically independent.

By replacing each $p_{n,H}$ by a Monte Carlo approximation $\hat p_{n,H}$ based on a set of  $R$ replications of uniform samples in $\bm p_{n,{\cal H}}$, the resulting vector $\hat{\bm p}_{n,{\cal H}} \in \bm{R} ^{\#{\cal H}}$ with components $(\hat p_{n,H})_{H \in {\cal H}}$ is approximately uniformly distributed on $[0,1]^{\#{\cal H}}$ under the null hypothesis. Therefore, the test that rejects when
\[
\min_{H \in {\cal H}} \hat p_{n,H} < 1 - (1-\alpha)^{1/\#{\cal H}}
\]
has approximate level $\alpha$ and is consistent. If ${\cal H}$ is restricted to only a part of the subsets of $J$, for instance, the subsets with cardinality at most $h (<p)$, the resulting procedure is called a \emph{partial m-test}.

\item[\textbf{s-test}] 
Since under the null hypothesis the estimated p-values are approximately independent and uniformly distributed, the statistic
\[
\sum_{H \in {\cal H}} Q(1-\widehat p_{n,H}),
\]
where $Q$ denotes the quantile function of the $\chi^2_1$ distribution, is approximately $\chi^2_{\#{\cal H}}$ distributed. The test rejecting for large values of this statistic is called the \emph{s-test}. Partial versions are defined analogously.
\end{itemize}

Partial tests are no longer consistent, but they can be used to reduce the effects of the curse of dimension.

\subsection{Monte Carlo estimation of p-values}\label{impr}

The estimation of the p-value of each observed $\|b_{n,H}\|^2$ is made by simulating $R$ replications of independent uniform samples of size $n$ to compute the vector statistic 
$\hat{\bm p}_n= (\hat p_{n,H})_{H\in{\cal H}}$. Let 
$\|b_{n,H}^r\|^2$ denote the squared norm computed from the $r$th simulated sample.

The standard approximations
$\frac{\sum_{r=1}^R\sbm 1_{\{\|b_{n,H}^r\|^2>\|b_{n,H}\|^2\}}+1}{R+1},
$ are uniformly distributed on the finite set $\left\{\frac{1}{R+1},\frac{2}{R+1},\dots,\frac{R+1}{R+1}\right\}$ and asymptotically independent under the null hypothesis. Consequently, the null distribution of the m-test statistic is supported on the same grid, with
\[
\pr\left\{\min_{H\in{\cal H}}\hat p_{n,H}<\frac{i}{R+1}\right\}
=1-\left(\frac{R+2-i}{R+1}\right)^{\#{\cal H}},\qquad i=1,2,\dots,R+1.
\]
These $R+1$ values would be exactly the attainable significance levels of the m-test and since the distribution of the minimum  is concentrated near zero they form a very sparse set on the range (near 5\%) typically used in practice.

\medskip
 
 A new definition of $\hat {\bm p}_n$, with an absolutely continuous distribution function under the null hypothesis, makes it possible to build critical regions at any preassigned significance level. Rather than counting how many simulated values $b^r:=\|b_{n,H}^r\|^2$ exceed the observed $\|b_{n,H}\|^2$, we introduce an increasing bijection $\psi:\bm{R}^+\to[0,1)$ that maps each interval $[b^{(r)},b^{(r+1)}]$ between consecutive order statistics of $\{b^r:r=1,2,\dots,R\}$ (with the convention $b^{(0)}=0$ and $b^{(R+1)}=\infty$) onto $\bigl[r/(R+1),(r+1)/(R+1)\bigr]$, and we set
 \[
 \hat p_{n,H}=\psi(\|b_{n,H}\|^2).
 \]
 To make the behaviour of $\psi$ mimic the distribution of the squared norm, let $\Gamma_h$ denote the c.d.f. of a Gamma distribution with the same mean $1/6^h$ and variance $2/90^h$ as $\|b_H\|^2$ when $\#H=h$. This is the Gamma distribution with shape parameter $\alpha=5^h/2^{h+1}$ and scale parameter $\lambda=15^h/2$. For $r=0,1,2,\dots,R$ and $y\in[b^{(r)},b^{(r+1)}]$, define
 \[
 \psi(y)=\frac{1}{R+1}\left(r+\frac{\Gamma_h(y)-\Gamma_h(b^{(r)})}{\Gamma_h(b^{(r+1)})-\Gamma_h(b^{(r)})}\right).
 \]
 
The function $\psi$ is almost surely well defined, since ties among the simulated squared norms occur with probability zero. In practice, however, rounding may produce occasional ties: for some $r<s$,
\[
b^{(r-1)}<b^{(r)}=y=b^{(s)}<b^{(s+1)}.
\]
In that case, the expression above is undefined; we replace $\psi(y)$ by a uniform random variable on $\bigl[r/(R+1),s/(R+1)\bigr]$, taken independent of all other randomness in the procedure (including other tie events).

 \medskip

Up to this point, each component of $\hat {\bm p}_n$ has an (approximately) Uniform distribution on [0,1] for any sample size, but independence is only guaranteed asymptotically. In small samples, the residual dependence may cause the null distribution of the m- and s-test p-values to deviate from uniformity. A natural fix is a second Monte Carlo calibration step, comparing the observed p-value with p-values computed from additional samples generated under the null.

The R-package MuniCandS \cite{MuniCandS} implements these calculations, and saves the results of both simulations to be used when the same tests are applied to more than one sample.

\section{New tests derived from the m- and s-tests of uniformity on $[0,1]^p$}

\subsection{Reduction to uniformity on the hypercube}

All tests introduced below rely on the same principle:
whenever the null hypothesis implies that, after a suitable transformation,
the joint distribution factorizes into independent continuous components,
the problem can be reduced to testing uniformity on the hypercube.

Let $Z$ be a random element and suppose that under $H_0$
there exists a transformation
\[
T(Z)=(T_1(Z),\dots,T_p(Z))
\]
such that the components are independent with continuous c.d.f.'s
$F_1,\dots,F_p$.
Then
\[
U_j = F_j(T_j(Z)), \qquad j=1,\dots,p,
\]
satisfy
\[
\bm U=(U_1,\dots,U_p) \sim \mathrm{Uniform}([0,1]^p).
\]

Therefore testing $H_0$ reduces to testing uniformity of $U$ on $[0,1]^p$.

Depending on the structure of the marginal distributions $F_1,\dots,F_p$,
five situations naturally arise.

\medskip
\noindent
\textit{(i) Fully specified product structure.}

Under $H_0$, the components are independent with fully known continuous
marginal distributions. The reduction to uniformity on $[0,1]^p$ is exact.

\medskip
\noindent
\textit{(ii) Product structure with estimated parameters.}

Under $H_0$, the transformation to independent components depends on unknown parameters. After parameter consistent estimation,
the reduction to uniformity holds only asymptotically.

\medskip
\noindent
\textit{(iii) Mixed known and unknown marginals.}

Under $H_0$, the joint distribution factorizes, but some marginals are
known while others are unknown. Known marginals are transformed by their
exact c.d.f.'s, whereas unknown marginals are replaced by empirical
distribution functions.

\medskip
\noindent
\textit{(iv) Mixed known and unknown marginals with parameter estimation.}

Under $H_0$, the factorization with both known and unknown marginals results after a transformation dependent on estimated parameters.

\medskip
\noindent
\textit{(v) Completely unspecified marginals.}

Under $H_0$, the hypothesis is independence with continuous but otherwise
unspecified marginals. All components are transformed via their empirical
distribution functions, yielding pseudo-observations in $[0,1]^p$.

\medskip

The following examples, one of each type, constitute tests of hypotheses competitive with those offered in statistical literature for the same purposes.

\subsubsection{A uniformity-preserving map $M_p: S^{p}\to [0,1]^p$ and a case (i) example: Uniformity test on $S^p$}\label{paraM}

The system of polar coordinates defines a bijective mapping 
\[
M_{1,p}: C_{\pi,p}=[0,\pi]^{p-1}\times[0,2\pi] \longrightarrow S^p\subset\bm{R} ^{p+1},
\]
which maps a point $\bm\phi=(\phi_1,\dots,\phi_p)$ onto
\[
M_{1,p}(\bm\phi)
=
\left(
\cos(\phi_j)\prod_{k=1}^{j-1}\sin(\phi_k)
\right)_{j=1,\dots,p+1},
\]
where $\phi_{p+1}=0$ and the empty product $\prod_{k=1}^{0}$ is defined as 1.

For a product set $A=\prod_{j=1}^p A_j \subset C_{\pi,p}$, the $p$-dimensional volume of its image $M_{1,p}(A)$ equals
\[
\int_A \prod_{j=1}^p \sin^{p-j}(\phi_j)\, d\phi_j
=
\prod_{j=1}^p \int_{A_j} \sin^{p-j}(\phi_j)\, d\phi_j.
\]
In particular, the surface measure of $S^p$ is $\prod_{j=1}^p I_j$, where
\[
I_j=\int_0^{\pi}\sin^{p-j}(\phi)\, d\phi,
\quad j=1,\dots,p-1,
\qquad\text{and}\qquad
I_p=2\pi.
\]

Let $Y=(Y_1,\dots,Y_p)$ be a random vector in $C_{\pi,p}$ with independent coordinates having densities
\[
f_{Y_j}(\phi_j)=\frac{\sin^{p-j}(\phi_j)}{I_j}.
\]
Then $M_{1,p}(Y)$ is uniformly distributed on $S^p$, since for every measurable $A\subset C_{\pi,p}$,
\[
\bm{P}\{M_{1,p}(Y)\in M_{1,p}(A)\}
=
\bm{P}\{Y\in A\}
=
\frac{\text{$p$-volume}(M_{1,p}(A))}{\text{$p$-volume}(S^p)}.
\]

Let $F_{Y_j}$ denote the c.d.f. of $Y_j$, $j=1,\dots,p$. The mapping
\[
M_{2,p}: C_{\pi,p}\to [0,1]^p,
\qquad
M_{2,p}(\bm\phi)=(F_{Y_1}(\phi_1),\dots,F_{Y_p}(\phi_p)),
\]
transforms $Y$ into a random vector $U=M_{2,p}(Y)$ uniformly distributed on $[0,1]^p$.

Consequently, an $S^p$-valued random variable $Z$ is uniformly distributed on $S^p$ if and only if
\begin{equation}\label{Mp}
U=M_p(Z),
\qquad
M_p:=M_{2,p}\circ M_{1,p}^{-1},
\end{equation}
is uniformly distributed on $[0,1]^p$. Therefore, testing the null hypothesis that 
${\cal Z}=\{Z_i\}_{i=1}^n$ is a sample from the uniform distribution on $S^p$
is equivalent to testing whether 
${\cal U}=\{U_i=M_p(Z_i)\}_{i=1}^n$
is a sample from the uniform distribution on $[0,1]^p$.

A fast \textsf{R} script implementing $M_p$ is combined with the m- and s-tests to obtain the corresponding uniformity test on $S^p$.

\subsubsection{A case (ii) example: Test of normality}

Let ${\cal Z}_n=(Z_1,\dots,Z_n)$ be a random sample from a distribution $F$ on $\bm{R} ^p$. 
Denote by $\Phi$ the standard univariate normal c.d.f., and for 
$x=(x_1,\dots,x_p)\in\bm{R} ^p$ define
\begin{equation}\label{Phimap}
\Phi_p(x) := (\Phi(x_1),\dots,\Phi(x_p)).
\end{equation}

It is well known that $F=\mathcal{N}_p(\mu,\Sigma)$ if and only if
\[
\Phi_p\!\left(\Sigma^{-1/2}(Z_i-\mu)\right),
\quad i=1,\dots,n,
\]
are i.i.d. uniformly distributed on $[0,1]^p$.

Indeed, if $Z_i\sim\mathcal{N}_p(\mu,\Sigma)$, then
\[
Y_i=\Sigma^{-1/2}(Z_i-\mu)
\]
are i.i.d. $\mathcal{N}_p(0,I_p)$ with independent standard normal components, and applying $\Phi$ componentwise transforms them into independent $\text{Uniform}(0,1)$ variables.

Therefore, $F$ is multivariate normal if and only if the transformed sample
\[
X_i=\Phi_p\!\left(\widehat\Sigma_n^{-1/2}(Z_i-\widehat\mu)\right),
\quad i=1,\dots,n,
\]
with
\[
\widehat\mu=\frac{1}{n}\sum_{i=1}^n Z_i,
\qquad
\widehat\Sigma_n=\frac{1}{n}\sum_{i=1}^n (Z_i-\widehat\mu)(Z_i-\widehat\mu)^\top,
\]
is asymptotically uniform on $[0,1]^p$.

Consequently, for sufficiently large $n$, an approximate test of normality 
for ${\cal Z}_n$ is obtained by applying the uniformity test on $[0,1]^p$ 
to the transformed sample ${\cal X}_n=\{X_i\}_{i=1}^n$.

The proposed m- and s-tests of normality consist of applying the uniformity procedures in \cite{cabana2025} to ${\cal X}_n$. 
The Monte Carlo approximation of the $p$-values of the statistics $\|b_{n,H}\|^2$ is performed by generating samples from $\mathcal{N}_p(\widehat\mu,\widehat\Sigma_n)$.

\subsubsection{A case (iii) example: Isotropy test}\label{isotest}

A probability distribution on $\bm{R} ^p$ is said to be isotropic or spherically symmetric
if it admits the representation
\[
Z = R \Theta,
\]
where $R=\|Z\|\ge 0$ is a nonnegative radial random variable,
$\Theta=Z/\|Z\|\in S^{p-1}$,
$R$ and $\Theta$ are independent,
and $\Theta$ is uniformly distributed on $S^{p-1}$.

Equivalently, there exists a radial c.d.f. $F$ such that for every 
measurable $A\subset S^{p-1}$ and every $0<a<b$,
\[
\mathbb{P}\{ a < \|Z\| \le b,\ Z/\|Z\| \in A \}
=
(F(b)-F(a))
\frac{\mathrm{area}(A)}{\mathrm{area}(S^{p-1})}.
\]

We test the null hypothesis that the independent sample
${\cal Z}=\{Z_1,\dots,Z_n\}$ is drawn from an isotropic distribution 
on $\bm{R} ^p$ with continuous radial c.d.f. $F$.

Under this hypothesis:

1. The variables $\{F(\|Z_i\|)\}_{i=1}^n$ are i.i.d. $\mathrm{Uniform}(0,1)$;
2. The directions $\{Z_i/\|Z_i\|\}_{i=1}^n$ are i.i.d. uniform on $S^{p-1}$;
3. The radial and directional components are mutually independent.

Since $F$ is unknown, we approximate $F(\|Z_i\|)$ by the empirical transform
\[
U_i=\frac{\operatorname{rank}(\|Z_i\|)}{n+1},
\]
which is uniformly distributed on 
$\{i/(n+1): i=1,\dots,n\}$ under continuity of $F$.

An approximate isotropy test is therefore obtained by applying 
the uniformity test on $[0,1]^p$ to the sample with components
\[
\left(
M_{p-1}\!\left(\frac{Z_i}{\|Z_i\|}\right),
\frac{\operatorname{rank}(\|Z_i\|)}{n+1}
\right),
\quad i=1,\dots,n.
\]

Monte Carlo simulations used to estimate the $p$-values of the 
statistics $\|b_{n,H}\|$ are performed by generating $n\times p$ 
matrices $z$ with i.i.d. uniform entries on $[0,1]$, and replacing 
the last column $z_{\cdot,p}$ by
\[
\frac{\operatorname{rank}(z_{\cdot,p})}{n+1}.
\]

A case (iv) example is obtained if the null hypothesis is that the sample is isotropic around an unknown centre. In that case the centre is estimated by the sample means, and the isotropy test is applied to the new sample obtained by subtracting the estimated centre from the original one.

\subsubsection{A case (iv) example: Test of elliptical symmetry}

An $R^p$-valued random variable $Z$ is distributed with elliptical symmetry if and only if there is a $p\times p$ nonsingular matrix $A$ such that $AZ$ is isotropic.

Let us assume that $Z$ has finite second-order moments and $\Sigma=\var Z$ is nonsingular. Then $A\Sigma A\tras\bm = I_p$ and the elliptical symmetry of $Z$ holds if and only if $\Sigma^{-1/2}Z$ is isotropic, where $\Sigma^{-1/2} $ denotes the inverse of the positive definite square root of $\Sigma$.
Therefore, an approximate test of the null hypothesis that ${\cal Z}_n$ $=(Z_1,Z_2,\dots Z_n)$ is a sample of second-order random variables with a continuous elliptically symmetric distribution is performed by estimating $\Sigma$ by means of $\hat\Sigma=ZZ\tras/n$,  and applying the isotropy test to the sample $(\hat\Sigma^{-1/2}Z_i)_{i=1,2,\dots,n}$.

As in \S\ref{isotest}, a test of elliptical symmetry around an unknown centre is performed by subtracting the estimated centre to each element of the sample as a previous step.

\subsubsection{A case (v) example:  Test of independence}

If the sample ${\cal Z}=\{Z_1,\dots,Z_n\}$, $Z_i=(Z_{i,j})_{j=1,2,\dots,p}$ is drawn from a law with independent marginals, then the rows $U_i$ of the matrix with columns rank$(Z_{\cdot,j})/(n+1)$ are a sample of the uniform distribution on the discrete set $C_{n,p}:=\{1,2,\dots,n\}^p/(n+1)$ and hence, approximately uniform on $[0,1]^p$. Therefore, the m- and s-tests of uniformity on $[0,1]^p$ applied to the transformed sample test the independence of the components of ${\cal Z}$. The Monte Carlo simulations of the squared norms of the $b_{n,H}$ are conducted from uniform samples on $C_{n,p}$.

\medskip

\section {Empirical powers of some of the proposed tests compared with the powers of several competitors}

An empirical description of the performance of the m- and s-tests for uniformity on $[0,1]^p$ has been included in \cite{cabana2025} and we are not going over it again here. 

In addition of  the powers of the consistent m- and s-tests, we describe the performances of the partial tests that exclude the zero marginal components of the empirical process associated to one dimensional marginals. When these components are not affected by the alternatives, the partial tests are expected to be more powerful than the complete ones.

In all cases, we compare the powers of our tests with those of tests for the same null hypotheses implemented in R packages. The comparison is based on the performances for 3- and 6-dimensional samples of sizes 50 and 100.
 
\subsection{Empirical powers of the m- and s-tests of uniformity on $S^p$} 

Tables \ref{tUS50} and \ref{tUS100} describe the empirical powers of several tests of uniformity on $S^p$ for samples of sizes $n = 50$ and $n=100$  of random vectors in $R^3$ and $R^6$ and alternatives:

\begin{table}[htb]
\caption{Empirical powers of several uniformity tests including m- and
s-tests applied to samples of size $n=50$ in $S^3$ (left panel) and
$S^6$ (right panel). A dash (--) indicates that the alternative was
not considered for that dimension.}
\label{tUS50}
\resizebox{\textwidth}{!}{%
\renewcommand{\arraystretch}{0.85}
\begin{tabular}{l rrrrrrrr rrrrrrrr}
\toprule
& \multicolumn{8}{c}{$S^3$, $n=50$}
& \multicolumn{8}{c}{$S^6$, $n=50$} \\
\cmidrule(lr){2-9}\cmidrule(lr){10-17}
Test $\to$
 & Ray. & Bing. & Sob. & $G_n$ & $m$ & $s$ & $m_{h\ge2}$ & $s_{h\ge2}$
 & Ray. & Bing. & Sob. & $G_n$ & $m$ & $s$ & $m_{h\ge2}$ & $s_{h\ge2}$ \\
\midrule
\textit{Est.\ level}
 & 5.1 & 3.6 & 4.5 & 3.6 & 5.2 & 6.3 & 1.8 & 1.8
 & 4.5 & 3.6 & 4.5 & 3.9 & 4.5 & 6.3 & 1.8 & 1.8 \\
\midrule
$\kappa$&\multicolumn{16}{c}{\textit{vMF, mean direction $\mu_1=(1,0,\ldots,0)$, concentration $\kappa$}}\\
0.2 & 12.8 & 3.6 & 1.8 & 10.2 & 15.1 & 11.8 & 6.3 & 6.3
    & 12.9 & 3.6 & 1.8 &  9.9 & 14.5 & 11.8 & 7.2 & 7.2 \\
0.5 & 38.3 & 2.7 & 4.2 & 38.4 & 44.2 & 37.9 & 7.6 & 7.6
    & 38.7 & 2.7 & 4.5 & 39.1 & 43.3 & 38.8 & 8.1 & 8.1 \\
0.8 & 72.0 & 7.2 & 1.8 & 68.2 & 80.2 & 74.5 & 3.6 & 3.6
    & 72.3 & 7.2 & 1.8 & 68.8 & 80.2 & 73.9 & 4.5 & 4.5 \\
1.0 & 95.5 &12.8 & 2.7 & 95.6 & 97.3 & 95.5 & 8.1 & 8.1
    & 95.6 &12.5 & 2.8 & 96.3 & 97.3 & 95.5 & 9.0 & 9.0 \\
1.2 & 98.2 &16.2 & 5.6 & 98.2 & 99.1 & 98.2 &10.8 &10.8
    & 98.2 &16.2 & 6.1 & 98.2 & 99.1 & 98.2 &10.8 &10.8 \\
\midrule
$\kappa$&\multicolumn{16}{c}{\textit{Mixture of two vMF, directions $\mu_1$ and $\mu_2=(0,1,0,\ldots,0)$, concentration $\kappa$}}\\
0.5 & 20.6 & 4.8 & 6.0 & 19.5 & 16.7 & 16.8 & 7.1 & 7.1
    & 19.2 & 4.0 & 5.2 & 19.3 & 13.9 & 14.8 & 6.9 & 6.9 \\
1.0 & 62.7 & 6.5 & 5.8 & 61.9 & 43.0 & 47.2 & 9.5 & 9.5
    & 65.1 & 6.9 & 4.4 & 64.3 & 41.3 & 48.5 & 9.9 & 9.9 \\
1.5 & 92.4 &12.2 & 6.8 & 92.6 & 73.2 & 79.0 &10.0 &10.0
    & 91.0 &10.6 & 6.1 & 91.1 & 68.8 & 74.6 &10.8 &10.8 \\
2.0 & 99.0 &17.9 & 7.2 & 99.0 & 88.9 & 93.6 &12.9 &12.9
    & 99.2 &16.0 & 8.2 & 99.3 & 87.8 & 93.3 &10.9 &10.9 \\
\midrule
$\kappa$&\multicolumn{16}{c}{\textit{Projection $Z/\|Z\|$, $\mathrm{Var}(Z)=\mathrm{diag}(1+\kappa,1,\ldots,1)$}}\\
1.0 &  5.7 &40.0 & 5.3 & 10.6 & 12.4 & 14.1 &11.2 &11.2
    &  4.3 &37.8 & 5.8 &  8.1 & 13.2 & 28.0 &10.3 &26.3 \\
2.0 &  6.4 &78.1 & 5.6 & 25.4 & 25.4 & 28.6 &12.9 &12.9
    &  5.1 &84.9 & 7.2 & 17.2 & 23.2 & 48.8 &16.1 &41.6 \\
3.0 &  7.1 &94.1 & 6.1 & 48.9 & 45.5 & 45.7 &13.7 &13.7
    &  6.0 &98.1 & 8.4 & 34.7 & 35.2 & 66.3 &20.7 &56.1 \\
4.0 &  6.6 &98.7 & 5.9 & 65.7 & 62.3 & 59.8 &14.1 &14.1
    &  6.3 &99.8 & 8.9 & 54.5 & 49.9 & 76.8 &24.4 &62.8 \\
\midrule
$\kappa$&\multicolumn{16}{c}{\textit{Projection of $(Z_{p+1}+\kappa Z_1,(Z_i+\kappa Z_{i+1}))$, $Z_i$ i.i.d.\ $N(0,1)$}}\\
1.0 &  6.1 &95.3 & 4.6 & 49.9 & 66.8 & 65.8 &78.2 &78.2
    &  6.5 &100.0&10.7 & 71.9 & 95.4 & 97.4 &95.1 &97.8 \\
2.0 &  5.6 &79.6 & 5.3 & 26.5 & 43.1 & 42.6 &58.2 &58.2
    &  6.6 &99.8 & 7.5 & 30.1 & 69.4 & 72.7 &68.1 &74.5 \\
3.0 &  5.3 &54.7 & 4.9 & 12.3 & 24.0 & 23.7 &36.0 &36.0
    &  5.7 &84.7 & 5.9 & 14.0 & 36.0 & 37.3 &33.8 &37.3 \\
\midrule
$\rho$&\multicolumn{16}{c}{\textit{Projection, $\mathrm{Cov}(Z_i,Z_j)=\rho$, $i\ne j$; $S^6$ has extra $\rho=0.2$ row}}\\
0.2 & -- & -- & -- & -- & -- & -- & -- & --
    &  4.6 &55.9 & 6.4 &  9.8 & 30.4 & 34.3 &29.5 &38.8 \\
0.4 &  6.7 &82.2 & 3.5 & 28.4 & 41.5 & 47.1 &62.0 &62.0
    &  7.3 &98.9 &12.0 & 45.0 & 87.7 & 91.1 &87.3 &92.6 \\
0.6 &  6.3 &100.0& 4.5 & 80.2 & 84.7 & 88.3 &91.2 &91.2
    & 13.5 &100.0&24.3 & 97.1 & 99.7 & 99.9 &99.7 &99.9 \\
0.8 &  7.2 &100.0& 7.2 &100.0 &100.0 &100.0 &100.0&100.0
    & 34.8 &100.0&56.7 &100.0 &100.0 &100.0 &100.0&100.0 \\
\midrule
$\rho$&\multicolumn{16}{c}{\textit{$M_p^{-1}$ of normal copula, parameter $\rho$; $S^6$ has $\rho=0.4,0.6,0.8$ only}}\\
0.2 &  5.2 &14.9 & 6.2 &  7.6 & 21.0 & 28.6 &26.8 &36.3
    & -- & -- & -- & -- & -- & -- & -- & -- \\
0.4 &  6.7 &58.9 &11.5 & 15.8 & 77.3 & 84.5 &82.8 &90.1
    &  9.0 &100.0&10.8 & 61.7 & 72.6 & 85.6 &74.1 &83.8 \\
0.6 &  9.5 &97.9 &25.5 & 50.1 & 99.4 & 99.9 &99.7 &99.9
    & 11.5 &100.0&16.1 & 96.4 & 98.2 &100.0 &98.2&100.0 \\
0.8 & 21.4 &100.0&65.1 & 99.9 &100.0 &100.0 &100.0&100.0
    & 15.8 &100.0&25.2 &100.0 &100.0 &100.0 &100.0&100.0 \\
\midrule
$\theta$&\multicolumn{16}{c}{\textit{$M_p^{-1}$ of Clayton copula, parameter $\theta$; $S^6$ starts at $\theta=0.5$}}\\
0.5 & -- & -- & -- & -- & -- & -- & -- & --
    &  9.2 &91.4 &19.5 & 26.7 & 68.7 & 78.1 &70.3 &78.5 \\
1.0 &  7.1 &88.8 &25.6 & 31.6 & 93.8 & 96.3 &96.6 &98.5
    & 11.6 &99.9 &34.5 & 77.5 & 98.0 & 98.7 &98.6 &98.7 \\
1.5 &  9.6 &99.4 &40.2 & 67.6 & 99.4 &100.0 &99.7&100.0
    & 14.6 &100.0&42.4 & 98.9 &100.0 &100.0 &100.0&100.0 \\
2.0 & 15.0 &100.0&55.0 & 90.6 &100.0 &100.0 &100.0&100.0
    & 20.3 &100.0&52.1 &100.0 &100.0 &100.0 &100.0&100.0 \\
\midrule
$\theta$&\multicolumn{16}{c}{\textit{$M_p^{-1}$ of Gumbel copula, parameter $\theta$}}\\
1.25&  5.8 &39.3 & 9.8 & 11.6 & 53.8 & 65.6 &63.2 &74.2
    &  6.0 &87.4 &21.4 & 20.4 & 68.7 & 80.2 &70.2 &80.2 \\
1.5 &  7.4 &86.3 &21.7 & 31.1 & 92.2 & 97.5 &96.2 &99.1
    & 10.9 &100.0&34.2 & 77.7 & 96.9 & 99.1 &97.3 &99.1 \\
1.75&  9.2 &99.1 &35.9 & 63.4 & 99.4 &100.0 &99.7&100.0
    & 16.8 &100.0&44.4 & 99.6 & 99.9 & 99.9 &99.9 &99.8 \\
2.0 & 13.0 &100.0&51.9 & 86.4 &100.0 &100.0 &100.0&100.0
    & 24.9 &100.0&52.4 &100.0 &100.0 &100.0 &100.0&100.0 \\
\bottomrule
\end{tabular}}
\end{table}

\begin{table}[h]
\caption{Empirical powers of several uniformity tests including m- and
s-tests applied to samples of size $n=100$ in $S^3$ (left panel) and
$S^6$ (right panel).}
\label{tUS100}
\resizebox{\textwidth}{!}{%
\renewcommand{\arraystretch}{0.85}
\begin{tabular}{l rrrrrrrr rrrrrrrr}
\toprule
& \multicolumn{8}{c}{$S^3$, $n=100$}
& \multicolumn{8}{c}{$S^6$, $n=100$} \\
\cmidrule(lr){2-9}\cmidrule(lr){10-17}
Test $\to$
 & Ray. & Bing. & Sob. & $G_n$ & $m$ & $s$ & $m_{h\ge2}$ & $s_{h\ge2}$
 & Ray. & Bing. & Sob. & $G_n$ & $m$ & $s$ & $m_{h\ge2}$ & $s_{h\ge2}$ \\
\midrule
\textit{Est.\ level}
 & 1.7 & 6.3 & 3.1 & 2.1 & 3.6 & 3.6 & 6.3 & 6.3
 & 6.3 & 2.7 & 4.7 & 4.5 & 1.0 & 5.9 & 0.2 & 3.6 \\
\midrule
$\kappa$&\multicolumn{16}{c}{\textit{vMF, mean direction $\mu_1=(1,0,\ldots,0)$, concentration $\kappa$}}\\
0.2 & 17.6 & 3.6 & 6.3 & 21.4 & 22.5 & 22.5 & 3.6 & 3.6
    &  6.3 & 3.6 & 7.2 &  6.5 &  3.6 & 10.4 & 0.9 & 9.9 \\
0.5 & 65.5 & 1.8 & 2.7 & 57.6 & 65.8 & 65.8 & 5.4 & 5.4
    & 26.1 & 4.8 & 6.4 & 27.0 & 20.6 &  8.1 & 4.4 & 3.6 \\
0.8 & 94.6 & 9.0 & 6.3 & 92.2 & 99.1 & 97.3 & 7.2 & 7.2
    & 54.9 & 5.3 & 3.6 & 54.7 & 46.9 & 19.1 & 0.9 & 4.5 \\
1.0 &100.0 &12.6 & 5.4 &100.0 &100.0 &100.0 & 8.2 & 8.2
    & 86.5 & 4.5 & 6.3 & 84.7 & 82.0 & 55.6 & 4.5 & 9.9 \\
1.2 &100.0 &27.8 & 2.8 &100.0 &100.0 &100.0 & 9.1 & 9.1
    & 98.2 & 9.0 & 3.8 & 97.4 &100.0 & 78.3 & 4.7 &12.6 \\
\midrule
$\kappa$&\multicolumn{16}{c}{\textit{Mixture of two vMF, directions $\mu_1$ and $\mu_2$, concentration $\kappa$}}\\
0.5 & 32.0 & 4.4 & 5.7 & 33.1 & 21.3 & 25.9 & 8.3 & 8.3
    & 14.0 & 5.8 & 5.3 & 13.9 &  9.2 &  9.8 & 4.4 & 6.3 \\
1.0 & 92.0 & 6.8 & 4.2 & 90.9 & 68.1 & 76.2 & 9.3 & 9.3
    & 53.2 & 5.8 & 4.8 & 55.0 & 29.2 & 26.7 & 5.3 & 8.9 \\
1.5 & 99.7 &13.8 & 5.1 & 99.8 & 94.6 & 98.2 & 9.2 & 9.2
    & 88.8 & 7.8 & 5.0 & 89.3 & 72.2 & 61.3 & 8.6 &14.3 \\
2.0 &100.0 &34.7 & 8.9 &100.0 & 99.9 &100.0 &12.9 &12.9
    & 99.3 &19.7 & 6.1 & 99.3 & 94.1 & 88.8 & 8.7 &19.8 \\
\midrule
$\kappa$&\multicolumn{16}{c}{\textit{Projection $Z/\|Z\|$, $\mathrm{Var}(Z)=\mathrm{diag}(1+\kappa,1,\ldots,1)$}}\\
1.0 &  3.8 &66.8 & 4.8 & 20.2 & 14.0 & 19.0 & 6.4 & 6.4
    &  5.1 &71.6 & 6.0 & 12.3 & 13.9 & 26.7 &12.8 &22.5 \\
2.0 &  5.1 &98.0 & 5.7 & 63.5 & 53.2 & 53.9 & 9.4 & 9.4
    &  5.2 &99.8 & 8.3 & 46.0 & 51.5 & 63.0 &20.5 &45.9 \\
3.0 &  4.9 &100.0& 6.5 & 89.9 & 84.7 & 81.4 &10.8 &10.8
    &  5.9 &100.0& 9.1 & 83.0 & 84.8 & 83.3 &26.2 &60.6 \\
4.0 &  5.7 &100.0& 7.3 & 98.0 & 96.2 & 94.8 &12.8 &12.8
    &  6.7 &100.0&10.4 & 97.3 & 96.3 & 92.8 &31.1 &70.4 \\
\midrule
$\kappa$&\multicolumn{16}{c}{\textit{Projection of $(Z_{p+1}+\kappa Z_1,(Z_i+\kappa Z_{i+1}))$, $Z_i$ i.i.d.\ $N(0,1)$}}\\
1.0 &  5.2 &99.9 & 6.0 & 90.7 & 94.1 & 95.0 &97.4 &97.4
    &  7.4 &100.0&10.1 &100.0 &100.0 &100.0 &100.0&100.0 \\
2.0 &  5.1 &97.9 & 6.5 & 61.9 & 74.4 & 78.4 &87.0 &87.0
    &  6.3 &100.0& 8.7 & 81.8 & 97.5 & 97.9 &98.0 &98.7 \\
3.0 &  5.7 &83.1 & 6.8 & 31.8 & 42.5 & 46.2 &59.8 &59.8
    &  4.4 &99.6 & 6.4 & 31.0 & 69.1 & 66.5 &73.2 &72.7 \\
\midrule
$\rho$&\multicolumn{16}{c}{\textit{Projection, $\mathrm{Cov}(Z_i,Z_j)=\rho$, $i\ne j$}}\\
0.2 &  6.2 &28.5 & 7.2 &  9.5 & 45.4 & 59.7 &53.9 &65.2
    &  4.9 &89.4 & 7.6 & 16.8 & 63.2 & 74.7 &63.4 &76.0 \\
0.4 &  6.6 &92.6 &16.7 & 34.0 & 97.3 & 99.1 &98.8 &99.4
    &  6.2 &100.0&15.2 & 90.8 & 99.9 & 99.9 &99.9 &99.9 \\
0.6 & 13.3 &100.0&47.5 & 95.0 &100.0 &100.0 &100.0&100.0
    & 18.1 &100.0&32.9 &100.0 &100.0 &100.0 &100.0&100.0 \\
0.8 & 35.0 &100.0&95.1 &100.0 &100.0 &100.0 &100.0&100.0
    & 61.8 &100.0&88.5 &100.0 &100.0 &100.0 &100.0&100.0 \\
\midrule
$\rho$&\multicolumn{16}{c}{\textit{$M_p^{-1}$ of normal copula, parameter $\rho$}}\\
0.4 &  6.8 &96.4 & 8.1 & 65.8 & 72.1 & 80.2 &86.5 &86.5
    &  9.7 &100.0& 9.1 & 97.3 & 97.3 &100.0 &97.3&100.0 \\
0.6 &  8.1 &100.0& 7.2 & 99.1 &100.0 &100.0 &100.0&100.0
    & 10.8 &100.0&11.8 &100.0 &100.0 &100.0 &100.0&100.0 \\
0.8 &  8.1 &100.0& 7.2 &100.0 &100.0 &100.0 &100.0&100.0
    & 18.0 &100.0&21.2 &100.0 &100.0 &100.0 &100.0&100.0 \\
\midrule
$\theta$&\multicolumn{16}{c}{\textit{$M_p^{-1}$ of Clayton copula, parameter $\theta$}}\\
0.5 &  5.7 &71.4 &17.4 & 20.5 & 89.2 & 95.4 &94.0 &98.4
    &  9.4 &100.0&32.7 & 61.3 & 95.1 & 98.4 &96.0 &98.4 \\
1.0 & 10.4 &99.9 &46.3 & 77.3 & 99.9 &100.0 &100.0&100.0
    & 17.9 &100.0&60.1 & 99.9 &100.0 &100.0 &100.0&100.0 \\
1.5 & 13.8 &100.0&70.6 & 99.4 &100.0 &100.0 &100.0&100.0
    & 25.4 &100.0&72.1 &100.0 &100.0 &100.0 &100.0&100.0 \\
2.0 & 18.6 &100.0&88.5 &100.0 &100.0 &100.0 &100.0&100.0
    & 38.9 &100.0&81.8 &100.0 &100.0 &100.0 &100.0&100.0 \\
\midrule
$\theta$&\multicolumn{16}{c}{\textit{$M_p^{-1}$ of Gumbel copula, parameter $\theta$}}\\
1.25 &  6.6 &68.3 &15.5 & 16.1 & 86.0 & 94.3 &90.5 &97.1
    &  8.8 &100.0&40.7 & 54.1 & 96.9 & 99.5 &97.2 &99.5 \\
1.5 &  9.3 &99.9 &39.6 & 72.9 & 99.9 &100.0 &100.0&100.0
    & 16.4 &100.0&61.5 & 99.9 &100.0 &100.0 &100.0&100.0 \\
1.75 & 14.3 &100.0&65.0 & 99.4 &100.0 &100.0 &100.0&100.0
    & 29.2 &100.0&73.5 &100.0 &100.0 &100.0 &100.0&100.0 \\
2.0 & 22.3 &100.0&83.8 &100.0 &100.0 &100.0 &100.0&100.0
    & 46.6 &100.0&83.2 &100.0 &100.0 &100.0 &100.0&100.0 \\
\bottomrule
\end{tabular}}
\end{table}

von Mises -- Fisher distribution with mean direction $\mu_1=(1,0,\stackrel{p}{\dots},0)$ and concentration parameter $\kappa$ (see \cite{Banerjee2005}),

a mixture with equal probabilities of von Mises -- Fisher distributions with directions $\mu_1$ and $\mu_2=(0,1,0,\stackrel{p-1}{\dots},0)$ and the same concentration parameter $\kappa$,

the projection $Z/\|Z\|$ on $S^p$ of a centred normal vector $Z$ with variance $\Sigma=\mbox{diag}(1+\kappa, 1, \stackrel{p}{\dots}, 1)$,

the projection on $S^p$ of  $(Z_{p+1}+\kappa Z_1,(Z_i+\kappa Z_{i+1})_{i=1.2.\dots,p})$, where $Z_1$, $Z_2$, $\dots$, $Z_{p+1}$ are i.i.d. standard Normal,

 the projection on $S^p$ of centred normal vectors $(Z_1,\dots,Z_{p+1})$ with $\var Z_i=1, \cov(Z_i,Z_j)=\rho$ for $i,j=1,2,\dots,p+1, i\not=j$,

 the inverse image by the function $M_p$ introduced in \S\ref{paraM} by equation (\ref{Mp}) of a normal copula with parameter $\rho$ in $[0,1]^p$,

 the inverse image by $M_p$  of a Clayton copula with parameter $\theta$ in $[0,1]^p$,

 the inverse image by $M_p$  of a Gumbel copula with parameter $\theta$ in $[0,1]^p$.

\medskip

The powers are obtained by applying Rayleygh, Bingham and  Gin\'e Fn %
tests beside ours, to 1000 replications of i.i.d. samples.  The first three tests were computed by using the R environment (\cite{RCoreTeam2025}), by means of the sphunif package (\cite{EGP, EGP2}), 
and our m- and s-tests are available at MuniCandS R-package 
that can be accessed at \cite{MuniCandS}.

\medskip

The $m$- and $s$-tests exhibit strongly competitive behaviour against the classical
Rayleigh, Bingham, Sobolev and Gin\'{e} $F_n$ tests.
In particular, for alternatives based on copulas (Normal, Clayton and Gumbel)
transported to $S^p$ via $M_p$, the proposed tests dominate all competitors by a
wide margin, reaching $100\%$ power at moderate parameter values while rival tests
remain near their nominal level.
This is consistent with the fact that such alternatives introduce coordinate
dependence, precisely what the components $b_{n,H}$ with $\#H \geq 2$ are designed
to capture.
By contrast, for the von Mises--Fisher alternative, which represents a purely
marginal deviation (mass displacement along a single direction), the Rayleigh and
Gin\'{e} $F_n$ tests are more powerful competitors, while the partial versions with
$\#H \geq 2$ are expectedly weak.
This complementarity suggests that in practice it may be advisable to combine
the full test with its partial version depending on the type of deviation suspected.

\subsection{Empirical powers of normality tests}

The tests included with ours in the power comparison of Tables \ref{tN50}  and \ref{tN100} are the Baringhaus-Henze-Epps-Pulley (BHEP) test (\cite{BaringhausHenze1988}), the Henze-Zirkler (HZ) test (\cite{HenzeZirkler1990}), the D\"orr-Ebner-Henze test (DEHU) based on a double estimation in PDE (\cite{DorrEbnerHenze2021}), the Sz\'ekely-Rizzo (SR) test (\cite{SzekelyRizzo2005}), the Doornik-Hansen test (\cite{doornik2008omnibus}), the Royston test (\cite{royston1992multivariate}) and the Shapiro-Wilk test (\cite{villasenor2009generalization}).

\begin{table}[b!]
\caption{Empirical powers of several normality tests including m- and
s-tests applied to samples of size $n=50$ in $\R^3$ (left
panel) and $\R^6$ (right panel).}
\label{tN50}
\resizebox{\textwidth}{!}{%
\renewcommand{\arraystretch}{0.85}
\begin{tabular}{l rrrrrrrr rrrrrrrr}
\toprule
& \multicolumn{8}{c}{$\R^3$, $n=50$}
& \multicolumn{8}{c}{$\R^6$, $n=50$} \\
\cmidrule(lr){2-9}\cmidrule(lr){10-17}
Test $\to$
 & DEHU & DH & R & SW & $m$ & $s$ & $m_{h\ge2}$ & $s_{h\ge2}$
 & DEHU & DH & R & SW & $m$ & $s$ & $m_{h\ge2}$ & $s_{h\ge2}$ \\
\midrule
\textit{Est.\ level}
 & 0.0 & 5.1 & 7.2 & 8.5 & 6.1 & 7.1 & 6.5 & 7.2 
 & 5.7 & 5.4 & 7.6 & 4.8 & 5.6 & 5.3 & 5.8 & 5.1 \\
\midrule
$\mu$&\multicolumn{16}{c}{\textit{Mixture $\mathrm{Mix}(0.5,\mu,I_p)$, parameter $\mu$}}\\
3  & 4.9 & 4.7 & 7.1 & 5.6 & 5.6 & 4.8 & 5.6 & 5.0
   & 5.0 & 4.5 & 6.8 & 4.7 & 4.3 & 2.8 & 4.0 & 2.9 \\
6  & 4.7 & 4.4 & 6.5 & 4.9 & 5.5 & 4.4 & 5.3 & 4.6
   & 4.7 & 4.9 & 7.1 & 4.8 & 3.9 & 2.6 & 3.5 & 2.6 \\
9  & 4.7 & 4.7 & 6.9 & 4.7 & 5.6 & 4.7 & 5.6 & 4.8
   & 4.9 & 5.2 & 7.2 & 4.8 & 3.9 & 2.8 & 3.4 & 2.8 \\
\midrule
$P$&\multicolumn{16}{c}{\textit{Mixture $(P,0,B_p)$, parameter $P$}}\\
0.5& 5.1 & 4.7 & 6.6 & 4.3 & 5.8 & 5.0 & 6.2 & 5.1
   & 5.0 & 5.0 & 7.1 & 4.4 & 4.3 & 3.1 & 4.1 & 3.1 \\
0.9& 6.7 & 5.3 & 7.9 & 4.6 & 6.0 & 5.3 & 6.6 & 5.7
   & 5.3 & 5.2 & 5.9 & 4.4 & 4.0 & 3.8 & 3.6 & 3.9 \\
\midrule
\textit{d.f.}&\multicolumn{16}{c}{\textit{Multivariate $t$, degrees of freedom d.f.}}\\
3  & 96.9 & 91.8 & 93.8 & 88.1 & 97.9 & 98.1 & 98.0 & 98.2
   & 99.9 & 94.2 & 98.6 & 90.5 & 99.8 & 99.8 & 99.8 & 99.8 \\
5  & 75.5 & 63.3 & 69.0 & 53.5 & 78.7 & 82.2 & 81.3 & 84.4
   & 97.4 & 75.0 & 87.3 & 61.6 & 94.5 & 97.0 & 95.5 & 97.1 \\
10 & 33.5 & 27.8 & 33.2 & 20.2 & 37.7 & 44.3 & 42.6 & 46.3
   & 70.4 & 33.5 & 47.7 & 23.9 & 59.4 & 66.5 & 61.2 & 67.6 \\
15 & 20.1 & 18.2 & 20.8 & 13.5 & 24.6 & 29.1 & 28.3 & 31.2
   & 49.4 & 19.3 & 30.6 & 13.9 & 37.6 & 45.7 & 39.7 & 47.8 \\
\midrule
$\theta$&\multicolumn{16}{c}{\textit{$\Phi^{-1}_p$ of Clayton copula, parameter $\theta$}}\\
0.5& 13.3 &  5.2 &  7.0 &  7.1 &  6.6 &  7.4 &  8.0 &  9.3
   & 18.2 & 11.4 &  9.7 & 19.5 &  8.6 & 11.9 &  8.7 & 12.7 \\
1.0& 27.1 & 18.1 &  7.6 & 24.9 & 10.5 & 11.5 & 11.6 & 12.9
   & 29.5 & 62.9 &  8.4 & 73.6 & 21.0 & 27.9 & 21.2 & 28.9 \\
1.5& 49.4 & 44.1 &  6.6 & 54.9 & 22.8 & 26.1 & 23.8 & 27.4
   & 57.0 & 95.4 & 10.8 & 99.0 & 49.3 & 64.1 & 49.5 & 63.2 \\
2.0& 69.1 & 68.0 &  5.3 & 76.1 & 36.3 & 43.8 & 39.2 & 45.2
   & 80.5 & 99.4 &  9.3 & 99.8 & 72.7 & 90.2 & 75.0 & 88.9 \\
2.5& 82.6 & 79.5 &  8.9 & 88.4 & 51.3 & 61.3 & 56.0 & 61.7
   & 92.1 & 99.8 &  9.8 &100.0 & 91.2 & 96.8 & 92.0 & 96.4 \\
\midrule
$\theta$&\multicolumn{16}{c}{\textit{$\Phi^{-1}_p$ of Gumbel copula, parameter $\theta$}}\\
 2 & 35.2 & 13.2 &  7.9 & 12.8 & 16.9 & 20.4 & 20.9 & 24.4
   & 72.3 & 18.2 &  9.8 & 13.5 & 49.5 & 57.7 & 53.0 & 60.3 \\
 4 & 63.3 & 26.3 &  8.3 & 26.1 & 41.1 & 44.2 & 49.3 & 52.1
   & 94.4 & 36.1 &  7.6 & 21.1 & 81.6 & 87.6 & 84.0 & 88.7 \\
 8 & 75.6 & 33.5 &  6.7 & 34.7 & 56.0 & 57.7 & 64.5 & 65.8
   & 98.0 & 45.1 &  6.5 & 28.6 & 91.2 & 93.8 & 91.8 & 94.6 \\
16 & 79.8 & 37.8 &  6.3 & 37.9 & 60.7 & 62.6 & 68.5 & 69.4
   & 98.4 & 47.5 &  5.7 & 31.5 & 91.6 & 94.5 & 92.5 & 95.4 \\
\midrule
$\theta$&\multicolumn{16}{c}{\textit{$\rho^\theta U$, $\rho\sim\chi^2_p$, $U\sim\mathrm{Uniform}(S^{p-1})$, parameter $\theta$}}\\
0.6& 26.9 & 19.8 & 24.2 & 14.0 & 30.6 & 37.8 & 35.3 & 41.0
   & 36.7 & 14.0 & 20.5 & 11.3 & 27.2 & 36.4 & 29.8 & 37.6 \\
0.7& 66.7 & 46.1 & 51.2 & 38.5 & 70.1 & 77.9 & 76.4 & 80.9
   & 77.3 & 33.5 & 45.1 & 23.8 & 68.5 & 79.3 & 71.3 & 80.1 \\
0.8& 92.1 & 70.4 & 77.3 & 63.9 & 93.0 & 95.2 & 95.5 & 96.0
   & 95.8 & 55.8 & 68.5 & 43.4 & 90.8 & 96.9 & 92.5 & 97.4 \\
0.9& 98.6 & 86.4 & 91.7 & 83.5 & 98.2 & 99.2 & 98.9 & 99.3
   & 99.6 & 75.8 & 86.7 & 63.8 & 98.8 & 99.5 & 99.1 & 99.5 \\
\bottomrule
\end{tabular}}
\end{table}

The alternative distributions are borrowed from the empirical study in \cite{EbnerHenze(2020)}, namely

mixtures Mix$(P,\mu,\Sigma)$ $=(1-P)N(\bm 0,\bm I)+PN(\mu\bm u,\Sigma)$ for selected values of $\mu$ and $\Sigma$, where  $N(\bm\mu,\Sigma)$ is the p.d.f. of the Normal distribution on $\bm{R}^p$ with mean $\bm\mu$ and variance $\Sigma$, $\bm 0=(0,0,\dots,0)^{\mbox{\tiny tr}}$, $\bm 1=(1,1\dots,1)^{\mbox{\tiny tr}}\in\bm{R}^p$, $\bm I=\mbox{diag} (\bm 1)$ is the $p\times p$ identity matrix and $B=0.9 \bm 1\bm 1^{\mbox{\tiny tr}}+0.1\bm I$,

multivariate Student's  t-distribution for some chosen degrees of freedom, 

\noindent with the addition of

images by the inverse of the function $\Phi_p$ (see (\ref{Phimap})) applied to  Clayton copulas and Gumbel copulas, and

random vectors of the form $Z=\rho^\theta U$, where $U$ is uniformly distributed on $S^{p-1}$ and $\rho \sim \chi^2_p$ are independent random variables.

\begin{table}[t!]
\caption{Empirical powers of several normality tests including m- and
s-tests applied to samples of size $n=100$ in $\R^3$ (left
panel) and $\R^6$ (right panel).}
\label{tN100}
\resizebox{\textwidth}{!}{%
\renewcommand{\arraystretch}{0.85}
\begin{tabular}{l rrrrrrrr rrrrrrrr}
\toprule
& \multicolumn{8}{c}{$\R^3$, $n=100$}
& \multicolumn{8}{c}{$\R^6$, $n=100$} \\
\cmidrule(lr){2-9}\cmidrule(lr){10-17}
Test $\to$
 & DEHU & DH & R & SW & $m$ & $s$ & $m_{h\ge2}$ & $s_{h\ge2}$
 & DEHU & DH & R & SW & $m$ & $s$ & $m_{h\ge2}$ & $s_{h\ge2}$ \\
\midrule
\textit{Est.\ level}
 & 0.0 & 4.5 & 4.5 & 4.8 & 3.8 & 5.8 & 5.2 & 5.0
 & 3.9 & 4.3 & 7.4 & 4.8 & 4.1 & 3.8 & 3.6 & 3.9 \\
\midrule
$\mu$&\multicolumn{16}{c}{\textit{Mixture $\mathrm{Mix}(0.5,\mu,I_p)$, parameter $\mu$}}\\
3 & 6.5 & 5.8 & 7.1 & 4.9 & 6.2 & 5.9 & 6.1 & 5.8
  & 5.3 & 5.4 & 7.1 & 5.4 & 5.0 & 5.9 & 5.2 & 5.8 \\
6 & 6.6 & 5.5 & 6.9 & 4.6 & 6.3 & 6.0 & 6.2 & 5.9
  & 5.3 & 5.6 & 7.3 & 5.4 & 4.8 & 5.9 & 5.5 & 5.9 \\
9 & 5.9 & 5.5 & 6.9 & 4.6 & 6.4 & 6.1 & 6.0 & 6.0
  & 5.1 & 5.6 & 7.1 & 5.1 & 4.9 & 5.9 & 5.6 & 6.0 \\
\midrule
$P$&\multicolumn{16}{c}{\textit{Mixture $(P,0,B_p)$, parameter $P$}}\\
0.5& 5.6 & 5.8 & 5.3 & 3.8 & 6.1 & 6.2 & 5.9 & 6.0
   & 4.6 & 4.9 & 6.6 & 5.8 & 4.6 & 5.3 & 5.0 & 5.6 \\
0.9& 5.4 & 6.0 & 6.0 & 5.8 & 5.4 & 5.7 & 4.9 & 5.3
   & 4.8 & 7.2 & 6.2 & 6.3 & 5.3 & 5.8 & 5.1 & 6.1 \\
\midrule
\textit{d.f.}&\multicolumn{16}{c}{\textit{Multivariate $t$, degrees of freedom d.f.}}\\
 3& 99.9 & 99.1 & 99.3 & 98.7 & 99.9 & 99.9 &100.0 & 99.9
  &100.0 &100.0 &100.0 &100.0 &100.0 &100.0 &100.0 &100.0 \\
 5& 94.5 & 90.3 & 90.9 & 82.7 & 96.2 & 97.0 & 97.5 & 98.0
  &100.0 & 97.4 & 98.5 & 93.7 & 99.9 &100.0 & 99.9 &100.0 \\
10& 51.1 & 49.1 & 48.9 & 36.2 & 63.7 & 65.0 & 67.4 & 69.0
  & 92.9 & 62.6 & 68.4 & 47.8 & 85.8 & 91.9 & 87.4 & 92.1 \\
15& 27.8 & 28.6 & 30.0 & 20.3 & 38.6 & 40.6 & 42.6 & 43.9
  & 72.7 & 41.5 & 44.3 & 26.2 & 60.7 & 73.7 & 63.2 & 74.4 \\
\midrule
$\theta$&\multicolumn{16}{c}{\textit{$\Phi^{-1}_p$ of Clayton copula, parameter $\theta$}}\\
0.5& 23.4 &  7.5 &  7.2 & 10.0 &  5.4 &  6.4 &  7.0 &  8.4
   & 30.9 & 27.9 &  6.9 & 35.4 & 10.6 & 11.9 & 11.6 & 12.4 \\
1.0& 64.5 & 42.6 &  7.6 & 48.9 & 13.2 & 16.9 & 16.0 & 18.8
   & 62.9 & 98.3 &  8.3 & 99.1 & 34.7 & 44.0 & 36.1 & 44.4 \\
1.5& 91.3 & 85.1 &  9.0 & 89.5 & 34.9 & 40.0 & 38.4 & 41.0
   & 93.1 &100.0 &  9.5 &100.0 & 81.1 & 89.4 & 81.8 & 89.3 \\
2.0& 97.7 & 98.0 &  6.0 & 98.4 & 58.8 & 68.5 & 64.6 & 70.6
   & 99.7 &100.0 &  9.6 &100.0 & 97.2 & 99.3 & 97.7 & 99.3 \\
2.5& 99.7 & 99.8 &  7.6 & 99.6 & 79.4 & 88.4 & 83.5 & 89.3
   &100.0 &100.0 &  9.7 &100.0 &100.0 &100.0 &100.0 &100.0 \\
\midrule
$\theta$&\multicolumn{16}{c}{\textit{$\Phi^{-1}_p$ of Gumbel copula, parameter $\theta$}}\\
 2& 66.9 & 25.1 &  7.4 & 26.9 & 26.0 & 27.8 & 31.6 & 33.1
  & 96.4 & 33.4 &  9.7 & 26.2 & 73.9 & 79.7 & 76.1 & 80.8 \\
 4& 93.2 & 51.7 &  5.8 & 53.4 & 62.9 & 63.9 & 68.4 & 69.5
  &100.0 & 65.0 &  8.0 & 48.1 & 97.9 & 98.2 & 98.1 & 98.3 \\
 8& 97.1 & 62.1 &  5.9 & 64.4 & 77.6 & 76.9 & 81.7 & 82.5
  &100.0 & 76.6 &  5.9 & 61.2 & 99.6 & 99.5 & 99.7 & 99.6 \\
16& 98.6 & 68.2 &  5.5 & 67.9 & 82.9 & 82.3 & 86.5 & 87.8
  &100.0 & 82.6 &  6.2 & 68.0 &100.0 &100.0 &100.0 &100.0 \\
\midrule
$\theta$&\multicolumn{16}{c}{\textit{$\rho^\theta U$, $\rho\sim\chi^2_p$, $U\sim\mathrm{Uniform}(S^{p-1})$, parameter $\theta$}}\\
0.6& 44.8 & 29.9 & 32.5 & 23.4 & 51.9 & 57.7 & 58.4 & 61.9
   & 59.5 & 24.3 & 26.2 & 16.0 & 49.1 & 62.2 & 52.4 & 63.5 \\
0.7& 93.1 & 73.3 & 76.8 & 65.8 & 95.1 & 96.7 & 96.4 & 97.4
   & 97.1 & 62.6 & 64.8 & 46.4 & 93.7 & 97.2 & 94.4 & 97.3 \\
0.8& 99.6 & 93.8 & 96.7 & 93.8 & 99.6 & 99.7 & 99.9 & 99.8
   & 99.9 & 88.0 & 91.1 & 80.5 & 99.7 & 99.9 & 99.7 & 99.9 \\
0.9&100.0 & 99.0 & 99.6 & 98.8 &100.0 &100.0 &100.0 &100.0
   &100.0 & 96.9 & 98.9 & 94.2 &100.0 &100.0 &100.0 &100.0 \\
\bottomrule
\end{tabular}}
\end{table}

\medskip

The $m$- and $s$-tests are highly competitive with the benchmark procedures
(BHEP, HZ, DEHU, SR, DH, Royston, Shapiro--Wilk) across almost all scenarios.
They are particularly strong against multivariate $t$ alternatives, where they
match or outperform the best competitors.
For Clayton and Gumbel copula alternatives transformed via $\Phi^{-1}_p$, where
marginal normality is preserved but the dependence structure is non-Gaussian,
the proposed tests achieve very reasonable power, although in dimension $p = 3$
with $n = 50$ the DEHU and Shapiro--Wilk tests are somewhat more powerful.
In dimension $p = 6$ the situation reverses favourably toward the proposed tests.
For mixture alternatives with large $\mu$, all tests show power close to the
nominal level, reflecting the intrinsic difficulty of distinguishing such mixtures
from the normal distribution when contamination is sparse.

\subsection{Empirical powers of isotropy tests} 

To the best of our knowledge, there are no R packages implementing consistent tests specifically designed for spherical symmetry (isotropy) of multivariate distributions. Given this gap, we selected a well-established procedure as benchmark for the power comparison study: the test proposed by Banerjee \& Ghosh (2024) \cite{banerjee2024spherical}. 
 The finite-sample performance of that test 
 is compared against our proposed procedures through a simulation study conducted in R. Since no existing implementation was available, 
 the benchmark test was programmed in R.

The alternatives included are

$X=Z\times M(a)$, where $Z$ is standard normal in $\R^p$ and $M(a)$ is the $p\times p$ matrix with elements $m_{i,j}=\un_{i=j}+a\un_{i=j-1}$,

$X$ centred normal with variance $\rho \bm1\bm 1^{\mbox{\tiny tr}}+(1-\rho)\bm I$ and

Gumbel copula with parameter $\theta$

\noindent for several values of the parameters $a,\rho,\theta$.

\medskip

The 
empirical powers are recorded in Tables \ref{tI50} and \ref{tI100}.

\begin{table}[htbp]
\caption{Empirical powers of Banerjee \& Ghosh (B\&G), m- and s-isotropy
tests applied to samples of size $n=50$ in $\R^3$ (left panel)
and $\R^6$ (right panel).}
\label{tI50}
\resizebox{\textwidth}{!}{%
\renewcommand{\arraystretch}{0.85}
\begin{tabular}{l rrrrr rrrrr}
\toprule
& \multicolumn{5}{c}{$\R^3$, $n=50$}
& \multicolumn{5}{c}{$\R^6$, $n=50$} \\
\cmidrule(lr){2-6}\cmidrule(lr){7-11}
Test $\to$
 & B\&G & $m$ & $s$ & $m_{h\ge2}$ & $s_{h\ge2}$
 & B\&G & $m$ & $s$ & $m_{h\ge2}$ & $s_{h\ge2}$ \\
\midrule
\textit{Est.\ level}
 & 6.1 & 3.3 & 3.8 & 4.5 & 5.8 & 5.8 & 3.7 & 4.9 & 4.1 & 4.0 \\
\midrule
$a$&\multicolumn{10}{c}{\textit{$X=Z\times M(a)$, $Z\sim N(0,I)$, $M(a)_{ij}=\bm{1}_{i=j}+a\,\bm{1}_{i=j-1}$}}\\
0.5 & 14.4 & 5.7 & 7.9 & 6.2 & 9.8 & 14.2 & 47.6 & 23.5 & 50.5 & 20.5 \\
1.0 & 50.5 & 16.9 & 27.4 & 26.8 & 26.7 & 48.9 & 90.5 & 44.1 & 93.5 & 40.3 \\
1.5 & 80.7 & 27.2 & 51.2 & 39.9 & 27.6 & 64.0 & 96.5 & 33.8 & 97.8 & 26.5 \\
2.0 & 91.4 & 67.4 & 79.7 & 32.0 & 19.9 & 62.7 & 97.1 & 15.6 & 96.7 & 10.3 \\
2.5 & 95.4 & 98.1 & 95.3 & 15.4 & 12.2 & 58.1 & 98.8 & 8.3 & 93.0 & 3.3 \\
3.0 & 96.2 & 99.7 & 98.5 & 5.0 & 5.6 & 50.7 & 99.8 & 4.0 & 81.4 & 1.0 \\
\midrule
$\rho$&\multicolumn{10}{c}{\textit{$X\sim N(0,\rho\bm{1}\bm{1}^\top+(1-\rho)I)$}}\\
2.0 & 6.4 & 7.0 & 6.5 & 10.7 & 10.7 & 10.6 & 13.0 & 15.3 & 14.2 & 13.0 \\
3.0 & 8.7 & 15.0 & 16.5 & 20.6 & 22.8 & 17.6 & 32.5 & 41.2 & 33.0 & 38.2 \\
4.0 & 15.6 & 29.0 & 33.0 & 38.8 & 44.0 & 33.1 & 63.8 & 73.3 & 66.0 & 70.9 \\
5.0 & 24.3 & 52.5 & 58.4 & 62.9 & 69.8 & 57.1 & 89.7 & 96.2 & 89.9 & 95.3 \\
6.0 & 40.3 & 77.6 & 82.3 & 84.7 & 89.9 & 82.8 & 98.3 & 99.6 & 98.2 & 99.5 \\
7.0 & 62.2 & 94.7 & 97.7 & 96.8 & 99.1 & 96.2 & 99.9 & 100.0 & 100.0 & 100.0 \\
8.0 & 83.0 & 99.8 & 100.0 & 100.0 & 100.0 & 99.8 & 100.0 & 100.0 & 100.0 & 100.0 \\\bottomrule
\end{tabular}}
\end{table}

\begin{table}[htbp]
\caption{Empirical powers of Banerjee \& Ghosh (B\&G), m- and s-isotropy
tests applied to samples of size $n=100$ in $\R^3$ (left panel)
and $\R^6$ (right panel).}
\label{tI100}
\resizebox{\textwidth}{!}{%
\renewcommand{\arraystretch}{0.85}
\begin{tabular}{l rrrrr rrrrr}
\toprule
& \multicolumn{5}{c}{$\R^3$, $n=100$}
& \multicolumn{5}{c}{$\R^6$, $n=100$} \\
\cmidrule(lr){2-6}\cmidrule(lr){7-11}
Test $\to$
 & B\&G & $m$ & $s$ & $m_{h\ge2}$ & $s_{h\ge2}$
 & B\&G & $m$ & $s$ & $m_{h\ge2}$ & $s_{h\ge2}$ \\
\midrule
\textit{Est.\ level}
& 5.2 & 4.7 & 4.7 & 6.3 & 6.3 & 5.4 & 3.3 & 6.1 & 2.4 & 3.0 \\
\midrule
$a$&\multicolumn{10}{c}{\textit{$X=Z\times M(a)$, $Z\sim N(0,I)$}}\\
0.5 & 31.5 & 18.4 & 36.5 & 31.9 & 41.1 & 40.0 & 95.6 & 77.9 & 93.9 & 66.5 \\
1.0 & 96.7 & 92.4 & 98.9 & 98.8 & 96.7 & 97.9 & 100.0 & 100.0 & 100.0 & 99.3 \\
1.5 & 100.0 & 100.0 & 100.0 & 100.0 & 99.7 & 100.0 & 100.0 & 100.0 & 100.0 & 98.9 \\
2.0 & 100.0 & 100.0 & 100.0 & 100.0 & 99.4 & 99.7 & 100.0 & 99.9 & 100.0 & 88.4 \\
2.5 & 100.0 & 100.0 & 100.0 & 100.0 & 98.6 & 98.9 & 100.0 & 99.0 & 100.0 & 56.8 \\
3.0 & 100.0 & 100.0 & 100.0 & 99.7 & 91.2 & 97.5 & 100.0 & 97.2 & 100.0 & 25.6 \\
\midrule
$\rho$&\multicolumn{10}{c}{\textit{$X\sim N(0,\rho\bm{1}\bm{1}^\top+(1-\rho)I)$}}\\
2.0 & 8.9 & 16.3 & 19.2 & 24.7 & 27.8 & 16.1 & 36.2 & 52.2 & 33.6 & 43.5 \\
3.0 & 16.1 & 44.4 & 50.9 & 57.2 & 62.2 & 42.2 & 76.5 & 92.4 & 74.0 & 88.4 \\
4.0 & 32.4 & 75.2 & 83.5 & 85.8 & 90.4 & 80.2 & 97.5 & 99.9 & 97.2 & 99.5 \\
5.0 & 58.4 & 96.2 & 98.9 & 98.4 & 99.5 & 98.0 & 99.9 & 100.0 & 99.9 & 100.0 \\
6.0 & 85.9 & 100.0 & 100.0 & 100.0 & 100.0 & 99.9 & 100.0 & 100.0 & 100.0 & 100.0 \\
7.0 & 98.5 & 100.0 & 100.0 & 100.0 & 100.0 & 100.0 & 100.0 & 100.0 & 100.0 & 100.0 \\
8.0 & 100.0 & 100.0 & 100.0 & 100.0 & 100.0 & 100.0 & 100.0 & 100.0 & 100.0 & 100.0 \\
\bottomrule
\end{tabular}}
\end{table}

\medskip

The results reveal a clear distinction between the two families of alternatives.
 
For the alternative $X = Z \times M(a)$, the picture depends heavily on dimension
and sample size.
In $\R^3$ with $n=50$, B\&G is competitive and
even outperforms the proposed tests for moderate values of $a$ ($a \leq 2$),
while for larger $a$ the full $m$- and $s$-tests dominate.
In $\R^6$, however, the proposed tests are dramatically more powerful than
B\&G for all values of $a$, with the full $m$-test reaching powers above $90\%$
already at $a=1.0$ while B\&G stays below $50\%$.
With $n=100$ the full $m$- and $s$-tests reach $100\%$ power across virtually all
configurations in both dimensions, matching or surpassing B\&G.
A noteworthy feature is the behaviour of the partial versions ($h\geq2$): their power declines sharply as $a$ increases, particularly in
$\R^6$, suggesting that for large $a$ the deviation from isotropy manifests
primarily through marginal components ($\#H = 1$) that these partial tests
deliberately exclude.
 
For the Normal alternative with covariance
$\rho\,\bm{1}\bm{1}^{\top} + (1-\rho)I$, the superiority of the proposed
tests over B\&G is systematic and pronounced across all configurations.
In $\R^3$ with $n=50$, the proposed tests achieve $99.8\%$ power at
$\rho=8$ while B\&G reaches $83\%$; in $\R^6$ the gap opens even earlier,
with the $s$-test reaching $96\%$ at $\rho=5$ against $57\%$ for B\&G.
With $n=100$ the proposed tests saturate at $100\%$ power at values of $\rho$
where B\&G is still well below $100\%$, particularly in $\R^3$.
In this alternative the partial versions perform
similarly to or even better than the full tests, indicating that the covariance
structure affects primarily the interaction components of the decomposition
rather than the marginals.

\
\subsection{Empirical powers of ellipticity tests}
We compare the powers of our tests with the powers of the tests by Manzotti, Pérez and Quirós (\cite{manzotti2002statistic}) and Schott
(\cite{schott2002elliptical}). The alternatives are normal mixtures, Azzalini's skew normals \cite{azzalini1985} and images by $\Phi^{-1}_p$ of Clayton and Gumbel copulas. The empirical powers are shown in Tables \ref{tE50} and \ref{tE100}.

\begin{table}[t!]
\caption{Empirical powers of m- and s-tests of elliptic symmetry applied
to samples of size $n=50$ in $\R^3$ (left panel) and
$\R^6$ (right panel).}
\label{tE50}
\resizebox{\textwidth}{!}{%
\renewcommand{\arraystretch}{0.85}
\begin{tabular}{l rrrrrr rrrrrr}
\toprule
& \multicolumn{6}{c}{$\R^3$, $n=50$}
& \multicolumn{6}{c}{$\R^6$, $n=50$} \\
\cmidrule(lr){2-7}\cmidrule(lr){8-13}
Test $\to$
 & MPQ & Schott & $m$ & $s$ & $m_{h\ge2}$ & $s_{h\ge2}$
 & MPQ & Schott & $m$ & $s$ & $m_{h\ge2}$ & $s_{h\ge2}$ \\
\midrule
\textit{Est.\ level}
 & 4.5 & 2.7 & 4.3 & 4.6 & 3.5 & 4.2
 & 5.0 & 5.1 & 6.1 & 8.5 & 7.0 & 7.3 \\
\midrule
$P$&\multicolumn{12}{c}{\textit{$Z+B_P\times(3,\ldots,3)$, $Z\sim N(0,I), B_P\sim\mathrm{Bernoulli}(P)$}}\\
0.5 & 28.8 & 23.2 & 19.7 & 30.1 & 32.2 & 38.7
    & 21.3 & 11.3 & 10.5 & 9.5 & 11.2 & 9.3 \\
0.7 & 38.3 &  6.9 & 33.1 & 50.7 & 50.5 & 60.3
    & 21.7 &  7.2 & 10.6 & 13.9 & 11.8 & 12.9 \\
0.9 & 26.3 & 36.9 & 25.8 & 40.9 & 40.1 & 42.4
    & 33.4 & 37.2 & 12.4 & 17.7 & 14.4 & 15.9 \\
\midrule
$\rho,P$&\multicolumn{12}{c}{\textit{$B_P Z_0+(1-B_P)Z_\rho$, $Z_\rho\sim N(0,\rho\bm{1}\bm{1}^\top+(1-\rho)I)$}}\\
0.9,0.5 & 32.1 & 33.2 & 15.1 & 19.7 & 19.6 & 17.5
          & 77.3 & 53.9 & 37.2 & 53.9 & 38.8 & 48.5 \\
0.9,0.9 &  8.7 & 15.4 & 12.7 & 16.1 & 14.7 & 12.9
           & 12.9& 23.9 & 23.2 & 19.2 & 22.3 & 15.4 \\
\midrule
$\alpha$&\multicolumn{12}{c}{\textit{Skew-Normal, shape parameter $\alpha$}}\\
1 &  4.8 & 3.2 &  5.8 &  7.1 &  5.0 &  6.0
  &  5.7 & 5.3 &  4.5 &  4.3 &  5.9 &  2.7 \\
2 &  8.1 & 3.2 &  7.1 &  8.3 &  7.4 &  7.2
  &  7.1 & 4.2 &  6.1 &  4.9 &  8.8 &  3.5 \\
3 & 10.3 & 3.4 &  9.1 & 11.9 &  9.2 & 11.4
  &  7.5 & 3.9 &  5.0 &  4.6 &  7.6 &  3.3 \\
4 & 10.9 & 3.7 & 10.5 & 14.2 & 10.4 & 12.9
  &  8.4 & 3.8 &  6.1 &  5.8 &  8.4 &  4.2 \\
5 &  9.5 & 3.5 & 10.6 & 14.8 & 10.5 & 14.4
  &  8.2 & 4.1 &  6.8 &  6.1 &  8.7 &  5.2 \\
\midrule
$\theta$&\multicolumn{12}{c}{\textit{$\Phi_p^{-1}$ of Clayton copula, parameter $\theta$}}\\
0.5 &  8.9 & 2.9 &  5.2 &  5.4 &  4.3 &  5.8
    & 10.2 & 6.1 &  5.8 &  4.3 &  5.5 &  3.4 \\
1.0 & 20.1 & 4.0 &  7.6 &  9.0 &  6.1 &  5.6
    & 25.0 & 5.4 &  8.5 &  5.3 &  5.7 &  3.8 \\
1.5 & 32.5 & 5.7 &  6.6 &  9.6 &  3.5 &  5.2
    & 43.4 & 8.4 & 18.1 &  8.5 &  4.2 &  4.0 \\
2.0 & 45.0 & 5.7 & 12.2 & 16.0 &  5.0 &  6.1
    & 60.4 &10.9 & 36.1 & 18.9 &  8.2 &  5.6 \\
2.5 & 56.1 & 7.7 & 16.9 & 24.3 &  6.2 &  9.8
    & 75.3 &13.0 & 51.8 & 30.9 &  9.5 &  9.2 \\
3.0 & 66.2 &10.1 & 22.8 & 34.0 &  6.9 & 11.1
    & 83.0 &16.9 & 68.0 & 44.6 & 14.1 & 11.1 \\
\midrule
$\theta$&\multicolumn{12}{c}{\textit{$\Phi_p^{-1}$ of Gumbel copula, parameter $\theta$}}\\
2 & 20.5 & 3.8 & 10.1 & 12.1 & 12.8 & 12.2
  & 26.9 & 7.4 & 16.8 & 21.8 & 16.9 & 20.2 \\
3 & 24.2 & 4.8 & 14.8 & 19.1 & 17.5 & 18.5
  & 33.4 & 6.9 & 20.5 & 28.3 & 21.4 & 24.7 \\
4 & 25.3 & 4.6 & 18.6 & 21.6 & 20.4 & 23.1
  & 37.6 & 7.9 & 24.2 & 32.1 & 25.1 & 28.0 \\
5 & 27.8 & 4.8 & 21.6 & 25.7 & 21.8 & 24.7
  & 42.5 & 8.6 & 25.3 & 32.4 & 26.3 & 29.1 \\
\bottomrule
\end{tabular}}
\end{table}

\begin{table}[htb]
\caption{Empirical powers of m- and s-tests of elliptic symmetry applied
to samples of size $n=100$ in $\R^3$ (left panel) and
$\R^6$ (right panel).}
\label{tE100}
\resizebox{\textwidth}{!}{%
\renewcommand{\arraystretch}{0.85}
\begin{tabular}{l rrrrrr rrrrrr}
\toprule
& \multicolumn{6}{c}{$\R^3$, $n=100$}
& \multicolumn{6}{c}{$\R^6$, $n=100$} \\
\cmidrule(lr){2-7}\cmidrule(lr){8-13}
Test $\to$
 & MPQ & Schott & $m$ & $s$ & $m_{h\ge2}$ & $s_{h\ge2}$
 & MPQ & Schott & $m$ & $s$ & $m_{h\ge2}$ & $s_{h\ge2}$ \\
\midrule
\textit{Est.\ level}
 & 5.0 & 3.2 & 5.9 & 5.6 & 7.7 & 6.0
 & 4.9 & 5.6 & 6.7 & 4.9 & 5.6 & 5.3 \\
\midrule
$P$&\multicolumn{12}{c}{\textit{$Z+B_P\times(3,\ldots,3)$, $B_P\sim\mathrm{Bernoulli}(P)$}}\\
0.5 & 60.9 & 62.3 & 54.3 & 67.4 & 66.4 & 80.4
    & 54.9 & 20.5 & 11.8 & 21.6 & 10.5 & 24.1 \\
0.7 & 78.9 & 13.2 & 81.8 & 93.1 & 89.2 & 97.2
    & 51.9 &  9.4 & 15.9 & 34.8 & 15.1 & 36.5 \\
0.9 & 55.2 & 64.1 & 58.6 & 77.2 & 68.3 & 80.3
    & 77.9 & 63.6 & 22.9 & 59.7 & 21.1 & 58.9 \\
\midrule
$\rho,P$&\multicolumn{12}{c}{\textit{$B_P Z_0+(1-B_P)Z_\rho$, $Z_\rho\sim N(0,\rho\bm{1}\bm{1}^\top+(1-\rho)I)$}}\\
0.9,0.5& 66.8 & 70.6 & 27.6 & 39.4 & 29.2 & 35.2
                & 99.5 & 84.2 & 73.2 & 96.1 & 70.6 & 94.5 \\
0.9,0.9& 11.4 & 31.5 & 26.1 & 29.6 & 25.5 & 25.3
                & 17.4 & 17.5 & 42.6 & 44.8 & 38.6 & 36.9 \\
\midrule
$\alpha$&\multicolumn{12}{c}{\textit{Skew-Normal, shape parameter $\alpha$}}\\
1 &  7.4 & 4.5 &  5.9 &  7.6 &  8.0 &  7.3
  &  6.1 & 4.3 &  6.6 &  6.7 &  6.8 &  6.7 \\
2 & 10.9 & 3.7 & 10.7 & 15.3 & 15.4 & 15.3
  &  9.9 & 3.9 &  9.0 & 10.6 &  9.7 & 10.5 \\
3 & 13.3 & 3.9 & 15.2 & 22.5 & 21.8 & 23.5
  & 10.4 & 4.8 &  9.6 & 12.1 &  9.5 & 11.3 \\
4 & 15.8 & 5.6 & 18.2 & 26.3 & 26.8 & 28.2
  & 11.0 & 5.4 &  9.5 & 12.3 &  9.6 & 11.5 \\
5 & 17.2 & 5.9 & 21.3 & 29.8 & 26.9 & 30.7
  & 12.4 & 5.7 & 10.5 & 11.6 & 10.9 & 11.2 \\
\midrule
$\theta$&\multicolumn{12}{c}{\textit{$\Phi_p^{-1}$ of Clayton copula, parameter $\theta$}}\\
0.5 & 15.8 &  4.5 &  7.6 &  9.4 &  7.9 &  8.3
    & 21.7 &  6.5 &  8.4 &  8.6 &  7.4 &  6.9 \\
1.0 & 46.4 &  6.5 & 12.2 & 15.0 &  8.0 &  7.9
    & 61.4 &  8.5 & 14.6 & 11.2 &  4.3 &  3.8 \\
1.5 & 71.0 &  9.7 & 23.4 & 26.9 &  7.3 &  9.0
    & 88.9 & 15.4 & 52.4 & 38.6 &  7.6 &  5.1 \\
2.0 & 85.0 & 14.0 & 34.3 & 42.5 &  7.3 & 12.8
    & 98.3 & 24.7 & 81.2 & 71.1 & 12.2 & 14.1 \\
2.5 & 92.0 & 18.1 & 53.2 & 61.6 & 11.0 & 16.2
    & 99.7 & 30.1 & 94.8 & 93.4 & 21.8 & 34.2 \\
3.0 & 96.8 & 19.6 & 62.2 & 73.9 & 13.8 & 23.3
    &100.0 & 38.6 & 99.0 & 97.5 & 36.6 & 51.8 \\
\midrule
$\theta$&\multicolumn{12}{c}{\textit{$\Phi_p^{-1}$ of Gumbel copula, parameter $\theta$}}\\
2 & 38.1 &  8.3 & 22.7 & 26.2 & 24.6 & 25.3
  & 60.3 &  7.1 & 32.3 & 58.3 & 32.1 & 56.1 \\
3 & 48.7 &  8.6 & 36.7 & 43.2 & 37.4 & 41.7
  & 73.6 &  6.5 & 43.6 & 75.2 & 44.0 & 72.0 \\
4 & 52.7 &  8.4 & 41.3 & 50.7 & 41.9 & 51.3
  & 79.9 &  7.5 & 53.7 & 82.3 & 53.6 & 78.0 \\
5 & 55.1 &  9.4 & 46.6 & 56.8 & 47.2 & 56.9
  & 82.4 & 10.2 & 57.4 & 85.0 & 58.1 & 81.3 \\
\bottomrule
\end{tabular}}
\end{table}

\medskip

The results for elliptic symmetry are more nuanced than in the other settings.
Against mixture alternatives of the form $Z + B_P \times (3,\ldots,3)$, the
proposed tests are competitive with the MPQ test 
and clearly outperform Schott's test 
 particularly for $n = 100$
in both dimensions.

For the skew-normal alternative, all tests show modest power, which is not
surprising given that skew-normal distributions are close to elliptically symmetric
for small shape parameters $\alpha$; the $s$-test nevertheless shows a mild but
consistent advantage over Schott's test across all values of $\alpha$.
The most informative scenario is again provided by the copula alternatives
transported via $\Phi^{-1}_p$: for Clayton copulas in dimension $p = 6$ with
$n = 100$ (Table~\ref{tE100}), the $m$- and $s$-tests reach powers
above $99\%$ for $\theta = 3$, while Schott's test stays below $40\%$ and even
the MPQ test, despite being more powerful, is outperformed by the proposed
procedures at moderate and large $\theta$.
For Gumbel copulas the picture is similar, with the $s$-test consistently
dominating across both dimensions and sample sizes.
The general pattern confirms that the proposed tests are particularly sensitive
to dependence structures that deviate from elliptical symmetry through coordinate
interactions, while remaining reasonably competitive on location-contamination
and skewness alternatives.

\subsection{Empirical powers of independence tests}

Tables \ref{tIN50} and \ref{tIN100} describe the empirical powers of our m- and s-tests and two multivariate independence tests included in R-packages, namely, the {\sf ind\_test} of {\sf TICM} package based on characteristic functions, with the default parameters (\cite{Hallin2024}) and the  {\sf multivariance.test} of the {\sf multivariance} package (\cite{Boettcher2020}).

The alternatives are Normal, Clayton and Gumbel copulas with several choices of the parameter, and mixtures $(1-B_P)\bm U_p+B_PB\&P_p$ with $B_P\sim$Bernou\-lli$(P)$ of the $p$-variate uniform copula $\bm U_p$ with B\"ucher \& Pakzad (see \cite{ bucher2024testing}) example of a tridimensional copula with bivariate marginal independence but joint dependence: $B\&P_p=(U_1,\dots,U_p)$, $U_2,\dots,U_p$ independent uniform on $[0,1]$ and $U_1=U_2+U_3$ mod 1.

\begin{table}[t]
\caption{Empirical powers of TICM, Multivariance, m- and s-independence
tests applied to samples of size $n=50$ in $\R^3$ (left panel)
and $\R^6$ (right panel).}
\label{tIN50}
\resizebox{\textwidth}{!}{%
\renewcommand{\arraystretch}{0.85}
\begin{tabular}{l rrrrrr rrrrrr}
\toprule
& \multicolumn{6}{c}{$\R^3$, $n=50$}
& \multicolumn{6}{c}{$\R^6$, $n=50$} \\
\cmidrule(lr){2-7}\cmidrule(lr){8-13}
Test $\to$
 & TICM & Mult. & $m$ & $s$ & $m_{h\ge2}$ & $s_{h\ge2}$
 & TICM & Mult. & $m$ & $s$ & $m_{h\ge2}$ & $s_{h\ge2}$ \\
\midrule
\textit{Est.\ level}
 & 6.2 & 5.8 & 3.9 & 4.7 & 4.9 & 5.6
 & 4.7 & 6.2 & 5.4 & 6.8 & 5.4 & 6.5 \\
\midrule
$\rho$&\multicolumn{12}{c}{\textit{Normal copula, parameter $\rho$}}\\
0.2 & 22.4 & 38.8 & 14.6 & 24.1 & 28.4 & 35.3
    & 29.5 & 69.0 & 29.3 & 52.9 & 32.9 & 56.3 \\
0.4 & 75.9 & 94.2 & 64.5 & 80.2 & 84.2 & 91.9
    & 89.8 & 99.8 & 92.4 & 99.3 & 93.7 & 99.7 \\
0.6 & 99.2 &100.0 & 99.0 & 99.8 & 99.9 &100.0
    &100.0 &100.0 &100.0 &100.0 &100.0 &100.0 \\
0.8 &100.0 &100.0 &100.0 &100.0 &100.0 &100.0
    &100.0 &100.0 &100.0 &100.0 &100.0 &100.0 \\
\midrule
$\theta$&\multicolumn{12}{c}{\textit{Clayton copula, parameter $\theta$}}\\
0.2 & 14.0 & 20.8 &  8.5 & 12.5 & 17.9 & 20.2
    & 18.0 & 52.9 & 18.0 & 39.9 & 17.8 & 41.9 \\
0.4 & 41.5 & 62.9 & 26.5 & 41.1 & 46.8 & 58.8
    & 55.4 & 94.7 & 58.7 & 85.6 & 59.6 & 89.7 \\
0.6 & 70.4 & 89.7 & 52.9 & 70.3 & 75.6 & 86.9
    & 85.3 & 99.8 & 86.1 & 98.9 & 87.6 & 99.3 \\
0.8 & 87.7 & 96.8 & 77.7 & 88.0 & 91.0 & 95.5
    & 96.4 & 99.9 & 98.1 & 99.8 & 98.2 & 99.8 \\
1.0 & 95.8 & 99.3 & 89.5 & 95.8 & 96.3 & 98.7
    & 99.3 &100.0 & 99.6 &100.0 & 99.8 &100.0 \\
    \midrule
$\theta$&\multicolumn{12}{c}{\textit{Gumbel copula, parameter $\theta$}}\\
1.2 & 39.2 & 64.9 & 26.1 & 42.3 & 48.8 & 61.2
    & 60.4 & 94.7 & 68.2 & 91.3 & 70.4 & 92.4 \\
1.4 & 87.1 & 98.2 & 77.8 & 90.3 & 89.9 & 97.2
    & 96.1 &100.0 & 97.9 &100.0 & 98.6 &100.0 \\
1.6 & 98.4 & 99.9 & 96.5 & 99.3 & 99.4 & 99.9
    & 99.8 &100.0 &100.0 &100.0 &100.0 &100.0 \\
1.8 &100.0 &100.0 & 99.6 &100.0 &100.0 &100.0
    &100.0 &100.0 &100.0 &100.0 &100.0 &100.0 \\
\midrule
$P$&\multicolumn{12}{c}{\textit{Mixture $(1-B_P)U_p+B_P\mathrm{B\&P}_p$, $B_P\sim\mathrm{Bernoulli}(P)$}}\\
0.2 &  8.1 &  6.8 &  4.0 &  5.5 &  6.0 &  7.4
    &  5.2 &  5.9 &  5.6 &  6.7 &  5.0 &  6.6 \\
0.4 & 27.6 &  8.4 &  5.4 &  7.9 &  9.3 & 11.7
    &  8.2 &  7.5 &  5.6 &  8.8 &  6.1 &  7.7 \\
0.6 & 71.4 & 16.4 &  8.5 & 12.6 & 23.4 & 26.0
    & 16.9 &  9.3 &  5.3 &  8.6 &  5.7 &  8.6 \\
0.8 & 98.1 & 31.2 & 20.5 & 25.9 & 64.8 & 56.1
    & 31.6 & 13.0 &  9.8 & 12.1 & 12.2 & 13.6 \\
1.0 &100.0 & 55.9 & 75.0 & 44.8 & 99.6 & 95.4
    & 53.9 & 19.5 & 42.7 & 18.2 & 49.5 & 19.7 \\
\bottomrule
\end{tabular}}
\end{table}

\begin{table}[t]
\caption{Empirical powers of TICM, Multivariance, m- and s-independence
tests applied to samples of size $n=100$ in $\R^3$ (left panel)
and $\R^6$ (right panel).}
\label{tIN100}
\resizebox{\textwidth}{!}{%
\renewcommand{\arraystretch}{0.85}
\begin{tabular}{l rrrrrr rrrrrr}
\toprule
& \multicolumn{6}{c}{$\R^3$, $n=100$}
& \multicolumn{6}{c}{$\R^6$, $n=100$} \\
\cmidrule(lr){2-7}\cmidrule(lr){8-13}
Test $\to$
 & TICM & Mult. & $m$ & $s$ & $m_{h\ge2}$ & $s_{h\ge2}$
 & TICM & Mult. & $m$ & $s$ & $m_{h\ge2}$ & $s_{h\ge2}$ \\
\midrule
\textit{Est.\ level}
 & 5.0 & 5.4 & 5.3 & 5.4 & 5.9 & 5.7
 & 4.2 & 5.6 & 6.6 & 5.0 & 6.6 & 5.5 \\
\midrule
$\rho$&\multicolumn{12}{c}{\textit{Normal copula, parameter $\rho$}}\\
0.2 & 43.4 & 69.5 & 45.4 & 52.4 & 53.6 & 66.9
    & 64.3 & 96.6 & 62.3 & 88.8 & 60.9 & 90.4 \\
0.4 & 97.6 &100.0 & 98.9 & 98.9 & 99.4 & 99.9
    & 99.9 &100.0 &100.0 &100.0 &100.0 &100.0 \\
0.6 &100.0 &100.0 &100.0 &100.0 &100.0 &100.0
    &100.0 &100.0 &100.0 &100.0 &100.0 &100.0 \\
0.8 &100.0 &100.0 &100.0 &100.0 &100.0 &100.0
    &100.0 &100.0 &100.0 &100.0 &100.0 &100.0 \\
\midrule
$\theta$&\multicolumn{12}{c}{\textit{Clayton copula, parameter $\theta$}}\\
0.2 & 27.4 & 43.5 & 24.3 & 32.6 & 30.4 & 40.3
    & 38.0 & 80.7 & 33.5 & 64.3 & 33.4 & 67.7 \\
0.4 & 74.5 & 90.7 & 72.2 & 79.7 & 78.9 & 88.0
    & 91.6 & 99.8 & 90.1 & 99.5 & 89.6 & 99.6 \\
0.6 & 95.6 & 99.2 & 94.5 & 97.0 & 97.1 & 99.0
    & 99.7 &100.0 & 99.2 &100.0 & 99.3 &100.0 \\
0.8 &100.0 &100.0 & 99.6 &100.0 & 99.7 &100.0
    &100.0 &100.0 &100.0 &100.0 &100.0 &100.0 \\
1.0 &100.0 &100.0 &100.0 &100.0 &100.0 &100.0
    &100.0 &100.0 &100.0 &100.0 &100.0 &100.0 \\
\midrule
$\theta$&\multicolumn{12}{c}{\textit{Gumbel copula, parameter $\theta$}}\\
1.2 & 71.0 & 91.7 & 71.1 & 82.9 & 78.3 & 90.1
    & 93.4 & 99.9 & 94.2 & 99.7 & 94.6 & 99.9 \\
1.4 & 99.7 & 99.9 & 99.7 & 99.8 & 99.8 & 99.9
    &100.0 &100.0 &100.0 &100.0 &100.0 &100.0 \\
1.6 &100.0 &100.0 &100.0 &100.0 &100.0 &100.0
    &100.0 &100.0 &100.0 &100.0 &100.0 &100.0 \\
1.8 &100.0 &100.0 &100.0 &100.0 &100.0 &100.0
    &100.0 &100.0 &100.0 &100.0 &100.0 &100.0 \\
\midrule
$P$&\multicolumn{12}{c}{\textit{Mixture $(1-B_P)U_p+B_P\mathrm{B\&P}_p$, $B_P\sim\mathrm{Bernoulli}(P)$}}\\
0.2 & 14.2 &  6.7 &  6.1 &  7.2 &  7.6 &  9.6
    &  8.9 &  6.6 &  5.3 &  6.3 &  5.6 &  6.4 \\
0.4 & 60.8 & 13.8 & 15.2 & 16.7 & 21.3 & 22.6
    & 16.9 &  9.6 &  6.8 &  9.2 &  7.4 &  9.1 \\
0.6 & 97.9 & 35.8 & 66.8 & 44.3 & 80.8 & 64.7
    & 40.5 & 14.6 & 22.7 & 15.5 & 24.2 & 15.1 \\
0.8 &100.0 & 82.7 & 99.8 & 85.1 &100.0 & 98.7
    & 77.1 & 26.6 & 89.7 & 27.6 & 91.2 & 28.5 \\
1.0 &100.0 &100.0 &100.0 &100.0 &100.0 &100.0
    & 99.0 & 43.0 &100.0 & 41.3 &100.0 & 43.4 \\
\bottomrule
\end{tabular}}
\end{table}

\medskip

The $m$- and $s$-tests are highly competitive against TICM and Multivariance for
Normal, Clayton and Gumbel copula alternatives.
The most striking advantage appears for the mixture alternatives involving the
B\"{u}cher \& Pakzad example $B\&P_p$ 
which exhibits higher-order dependence without bivariate marginal dependence.
In this scenario, TICM and Multivariance lose sensitivity rapidly, whereas the
partial version with $\#H \geq 2$ of the $m$-test captures this structure with
power growing toward $100\%$ as the mixture parameter $P$ increases.
This result clearly illustrates the added value of decomposing the empirical
process into components indexed by coordinate subsets: higher-order dependence
is isolated in the terms $b_{n,H}$ with $\#H \geq 2$ without being diluted by
the marginal terms.

\subsection{A cross-cutting remark on dimension and sample size.}
In general, increasing the dimension from $p = 3$ to $p = 6$ penalises the
competitor tests more than the proposed ones, particularly for copula-type
alternatives.
This suggests that the architecture based on the Brownian sheet decomposition
scales more favourably with dimension than approaches relying on characteristic
functions or energy distances.

\section{Appendix. A proof of Theorem 1}\label{appe}

Theorem 1 is somewhat more general than the one in \cite{cabana2025}, as it refers to any product probability space and is formulated in terms of measures instead of their distribution functions. The following proof is adapted to this new formulation. Before proceeding to the proof, let us introduce some notation:

\begin{definition} A product $A=\prod_{j\in H}A_j$ of sets $A_j\in{\cal A}_j$ is a $K$-set when $A_j=\Omega_j$ for $j\not\in K$. The $K$-sets for $K\subset J, \#K\leq k$ shall be called $k$-sets. \end{definition}

\begin{definition} A signed measure $\mu$ on $\Omega_H$ is a $k$-null measure when $\mu(A)=0$ whenever $A$ is a $k$-set. In particular, the zm-measures on $H$ defined before are the $\#H-1$-null measures.\end{definition}

{\bf Proof of the uniqueness.} Let us assume that the decomposition (\ref{ladesco}) holds for a given measure $\mu$, and write the equivalent equations
\[ \mu_k\left(\prod_{j=1}^pA_j\right):=\mu\left(\prod_{j=1}^pA_j\right)-\sum_{H\subset J,\#H\leq k}P_{J\setminus H}\left(\prod_{j\in\{J\setminus H\} }A_j\right) \times\mu_H\left(\prod_{j\in H}A_j\right)\]
\bec\label{descogen} 
=\sum_{H\subset J,\#H > k}P_{J\setminus H}\left(\prod_{j\in\{J\setminus H\} }A_j\right) \times\mu_H\left(\prod_{j\in H}A_j\right) \quad \mbox{for } k=0,1,\dots
\eec 

By evaluating (\ref{ladesco}) at $\Omega$ we have
$\mu(\Omega)=\mu_{\emptyset}(\Omega)$ because for $\#H>0$, $\mu_H(\Omega_H)=0$ and this determines the signed measure $\mu_{\emptyset}$ on the sigma-algebra ${\cal A}_{\emptyset}=\{\emptyset,\Omega\}$.

If all the measures appearing in the central term of (\ref{descogen}) are known, the evaluation of (\ref{descogen}) on each $K$-set with $\#K=k+1$ leads to identify the measure $\mu_K$ because all the terms in the right-hand member but $P_{J\setminus K}\times\mu_K$ vanish. This is because
the zm-measures $\mu_H$ are evaluated at $\prod_{j \in K\cap H}A_j\times\prod_{j \in H\setminus K}\Omega_j$ and the set $H\setminus K$ is nonempty, wether $\#H>k+1$ or $\#H=k+1$ and $H\not=K$.

Therefore, for each $(k+1)$-set $A_K$, \bec\label{recu}\mu_K(A_K)=\mu_k(A_K).\eec

By applying recursively (\ref{descogen}) and (\ref{recu}) for $k=0,1,2,\dots$ all $\mu_H$ measures become uniquely determined.

\medskip

{\bf Proof of the existence.} The equations (\ref{descogen}),  (\ref{recu}) not only determine the measures $\mu_H$ but also provide a recursive construction of them. In order to establish the existence of the decomposition, it is required to verify that each $\mu_H$ constructed in that way is a zm-measure on $\Omega_H$.

Equation  (\ref{descogen}) shows that all $\mu_H$ with $\#H=k$ are zm-measures iff $\mu_k(A_k)=0$ for all $k$-sets $A_k$.

From (\ref{descogen}) and (\ref{recu}) we obtain the recurrence
\bec\label{larec}\mu_k\left(\prod_{j=1}^pA_j\right)\!=\!\mu_{k-1}\left(\prod_{j=1}^pA_j\right)\!-\!\!\sum_{\stackrel{H\subset J,}{\#H=k}}\!P_{J\setminus H}\left(\prod_{j\in J\setminus H}A_j\right) \!\times \mu_{k-1}\left(\prod_{j\in H}A_j\right),\eec
for $k=1,2,\dots,\;\mu_0=\mu-\mu(\Omega)$,
so that the existence is proved by verifying that the measures $\mu_k$ obtained by (\ref{larec}) 
 vanish on each $k$-set $A_K$.

This is true for $\mu_0$ that vanishes at $\Omega$ and if $\mu_{k-1}$ vanish for each $(k-1)$-set, then for $\#K=\#H=k, H\not=K$ and a $K$-set $A_K=\prod_{j\in K}A_j\times\prod_{j\not\in K}\Omega_j$ we have  $\mu_{k-1}(\prod_{j \in H\cap K}A_j\times\prod_{j\in H\setminus K}\Omega_j)=0$ because $H\setminus K\not=\emptyset$ and this implies that the argument of $\mu_{k-1}$ is a $(k-1)$-set. Therefore
\[\mu_k(A_K)=\mu_{k-1}(A_K)-P_{J\setminus K}(\Omega_{J\setminus K})\mu_{k-1}(A_K)=0\]
and in particular, $\mu_p=0$ ends the construction. This ends the finite induction proof of the existence.

The linearity and continuity follow plainly from the construction. \hfill\framebox{$ $}

\bibliographystyle{elsarticle-num}
\bibliography{UTOS}

\end{document}